\documentclass{elsarticle}

\usepackage[english]{babel}
\usepackage{lineno}
\usepackage{numcompress}
\usepackage{amssymb}
\usepackage{amsmath}
\usepackage{latexsym}
\usepackage{amsthm}
\usepackage{fancyhdr}
\usepackage{graphicx}
\usepackage{newlfont}
\usepackage{textcomp}
\usepackage{graphicx}

\newtheorem{remark}{Remark}
\newtheorem{lemma}{Lemma}
\newtheorem{proposition}{Proposition}
\newtheorem{theorem}{Theorem}

\newtheorem{definition}{Definition}
\newtheorem{assumption}{Assumption}
\theoremstyle{definition}
\newtheorem{algorithm}{Algorithm}
\theoremstyle{plain} 

\usepackage[pdftex,
            pagebackref=true,
            colorlinks=true,
            linkcolor=blue,
            unicode,
            hidelinks
           ]{hyperref}
\usepackage{color}
\newcommand{\bornia}[1]{\textcolor{black}{#1}}
\newcommand{\giachi}[1]{\textcolor{black}{#1}}

\begin{document}

\begin{frontmatter}

\title{\bf{{Convergence estimates for multigrid algorithms with SSC smoothers and applications to overlapping domain decomposition}}}

\author{E. Aulisa \fnref{giacomoFootnote}}
\address{Department of Mathematics and Statistics, Texas Tech University}

\author{G. Bornia \fnref{giacomoFootnote}}
\address{Department of Mathematics and Statistics, Texas Tech University}

\author{S. Calandrini \fnref{giacomoFootnote}}
\address{Department of Mathematics and Statistics, Texas Tech University}

\author{G. Capodaglio \fnref{giacomoFootnote}\corref{giacomoCorresponding}}
\address{Department of Mathematics and Statistics, Texas Tech University}
\cortext[giacomoCorresponding]{Corresponding Author}
\ead{giacomo.capodaglio@ttu.edu}

\fntext[giacomoFootnote]{1108 Memorial Circle, Department of Mathematics and Statistics, Texas Tech University, Lubbock TX 79409, USA}

\begin{abstract}
In this paper we study convergence estimates for a multigrid algorithm with 
\bornia{smoothers of} \giachi{successive subspace correction (SSC) type},
applied to symmetric elliptic PDEs.
First, we revisit a general convergence analysis on a class of multigrid algorithms
in a fairly general setting, where no regularity assumptions are made on the solution.
In this framework, we are able to explicitly highlight the dependence of the multigrid error bound on the number of smoothing steps.
\giachi{For the case of no regularity assumptions, this represents a new addition to the existing theory.}
Then, we analyze \giachi{successive subspace correction} smoothing schemes
for a set of uniform and local refinement applications with either nested or non-nested overlapping subdomains.
For these applications, we explicitly derive bounds for the multigrid error,
\bornia{and identify sufficient conditions 
for these bounds to be independent of the number of multigrid levels}.
\giachi{For the local refinement applications, finite element grids with arbitrary hanging nodes configurations are considered.}
The analysis of these smoothing schemes is cast within the far-reaching \giachi{multiplicative Schwarz} framework.
\end{abstract}
 

\end{frontmatter}

\section{Introduction}


Multigrid algorithms have been introduced in the literature since the 1960's
with pioneering works such as \cite{fedorenko1962relaxation,bakhvalov1966convergence,nicolaides1975multiple,brandt1977multi}.
A wide literature of both theoretical and computational works has been developed ever since, 
 driven by appealing features such as optimal computational complexity \cite{brandt1977multi,yserentant1993old}.
Multigrid methods have been studied for the approximate solution of partial differential equations
in various discretization schemes, 
starting with finite differences \cite{fedorenko1962relaxation,nicolaides1975multiple,brandt1977multi} 
and then moving to finite elements \cite{nicolaides1977_l2,nicolaides1979some} and other settings.
The first results were obtained for elliptic operators of either symmetric \cite{nicolaides1977_l2}
or non-symmetric type \cite{bramble1994uniform, wang1993convergence, olshanskii2004convergence}.

Convergence proofs of multigrid algorithms usually rely on two properties 
referred to as the smoothing and the approximation property 
\cite{hackbusch2013multi, braess2007finite,bramble1987new}.
The former is related to the definition of the smoothing operator involved in the algorithm,
while the latter is usually proved assuming full elliptic regularity 
for the solution of the partial differential equation.
A breakthrough in the convergence analysis took place with
\cite{bramble1991convergence}.
In this work the elliptic regularity assumption has been dropped.
The error bound obtained is not optimal in the sense that
it becomes worse as the number of multigrid levels increases;
moreover, no dependence of the bound on the number of smoothing iterations is shown.
Further work has been done in this direction by Bramble and Pasciak \cite{bramble1993new}, 
where they showed
that optimal convergence can be obtained provided that \textit{partial} regularity assumptions are made.
However, no dependence on the number of smoothing iterations was reported yet.
In \cite{bramble1992analysis}, convergence estimates were obtained by the same authors for the case
of a multigrid algorithm with non-symmetric subspace correction smoothers under no regularity assumptions.
The error bound obtained for applications to both uniform and local refinement 
showed quadratic dependence on the total number of multigrid spaces 
but not on the number of smoothing steps.
An improvement in addressing this matter was made in \cite{brenner2002convergence}, 
where the author showed that the multigrid error bound is optimal 
and can be improved when increasing the number of smoothing iterations,
under partial regularity assumptions and using a Richardson relaxation scheme. 
\giachi{More recently, further work on multigrid methods that rely on minimal regularity assumptions
has been done in \cite{chen2012optimal}, where graded meshes obtained by a variant of the newest vertex bisection
method are considered.}

This work aims at first to further contribute to the description of multigrid methods 
by carrying out a general convergence analysis that does not require any regularity assumption. 
Only three clear assumptions
on the smoothing error operator are identified
which produce a multigrid error bound that shows dependence on the number of smoothing steps
and on the continuity constants of the smoothing error operators in the topology of the energy norm.
\giachi{As mentioned in the abstract, the explicit dependence of the multigrid error bound on
the number of smoothing iterations is a new result under no-regularity assumptions.}

The setup of this framework is then used to analyze
 smoothing schemes of \giachi{successive subspace correction (SSC)} type.
A unifying scheme encompassing successive subspace correction algorithms 
is given by Xu in \cite{xu1992iterative}.
See also \cite{bramble1992analysis} for an analysis of smoothers in a general framework 
to which subspace correction smoothers of either additive or multiplicative type belong.
We will study both uniform and local refinement applications \giachi{with arbitrary hanging nodes configurations},
to show under what conditions on the subdomain solvers multigrid convergence is achieved,
and when it is possible to obtain optimal multigrid error bounds, i.e.,
independent of the total number of levels.
For the uniform refinement case, such results upgrade the ones in \cite{bramble1992analysis}.
For the two local refinement applications, 
we derive ad-hoc decompositions of finite element spaces 
and set suitable choices of approximate subdomain solvers.
These are needed when dealing with hanging nodes that are introduced by the local refinement procedure.
In \giachi{the first local refinement} case, we construct a decomposition that has the advantage of being easy to implement and suitable for 
standard finite element codes,
but it does not allow freedom in the choice of the subdomains on which the subspaces are built.
A second decomposition requires additional work to ensure continuity of the finite element solution 
and so it requires a non-standard finite element implementation. 
However, it enables a choice of the subdomains that does not depend on the multigrid level.
\giachi{Numerical results for a similar choice of decomposition, together with a complexity analysis
of the resulting algorithm, have been provided in \cite{janssen2011adaptive}.
In such a work, the smoothing procedure is carried out only locally, 
rather than at all nodes of a given multigrid level, as we do in this theory.
A convergence analysis for this local smoothing approach is available in \cite{bramble1992analysis}.}
In both the aforementioned local refinement applications \giachi{studied in this paper}, 
the multigrid error bound shows a quadratic dependence on the multigrid level 
and this agrees with what was found in \cite{bramble1992analysis}. 
Furthermore, we explicitly show a dependence on the number of smoothing steps.
This allows us to identify conditions on the number of smoothing iterations 
that guarantee convergence as well as optimality of the error bound. 
Basically, these conditions establish a balance between the action of the smoothing error 
and the number of smoothing steps.
In the applications, we obtain smoothing error bounds 
that are either constant or increase tending to one with increasing level, 
thus corresponding to a poorer smoothing action with increasing level.
In order to have convergence, only one smoothing iteration at each level is sufficient.
Nevertheless, we find that optimality of the multigrid error bound
may be obtained only with a quadratically increasing number of smoothing steps.
Thus, the convergence deterioration of the smoother with an increasing number of levels 
may be compensated by an ad-hoc number of smoothing steps 
in order to obtain optimality.
\giachi{We remark that when a local smoothing procedure is conducted, increasing the number of smoothing iterations 
at a given multigrid level would only improve the multigrid error bound up to a given saturation value.
As a consequence, a deterioration of the error bound that goes with the total number of levels could not be
balanced by increasing the number of smoothing steps.}

  The outline of the paper is as follows. 
  In Section \ref{sec_mg} the multigrid algorithm is 
  described together with a general convergence theory.
  Such theory is based on three assumptions on the smoothing error operator and needs no regularity.
  Section \ref{sec_smoother} illustrates the algorithm used for the smoothing iteration 
  and shows how it can be related to the multigrid convergence theory.  
  Uniform and local refinement applications of the analysis described in the previous sections
  are presented in Section \ref{sec_app}, where convergence bounds are obtained for the specific cases.
  Finally, we draw our conclusions.

 \section{The multigrid algorithm} \label{sec_mg}

 In this section we describe the multigrid algorithm subject to our analysis.
  Throughout the paper, the total number of levels will be denoted as $J$.
  For $k = 0,  \dots  ,J$, let $V_k$ be a finite-dimensional vector space such that 
 \begin{equation} \label{mg_spaces}
V_0 \subset V_1 \subset \dots \subset V_J   \,,
 \end{equation}
 and let $ (\cdot,\cdot) $ and  $ a(\cdot,\cdot) $ be two symmetric positive definite (SPD) 
 bilinear forms on $ V_k $.
 Hence, both bilinear forms are inner products on $V_k$.
 Let $|| \, \cdot \, || = \sqrt{( \cdot , \cdot )}$  
 and $ || \, \cdot \, ||_E = \sqrt{a ( \cdot , \cdot )}$ be the corresponding induced norms.
 Associated with these inner products, 
 let us also define the operators $ Q_k : V_J \rightarrow V_k $ and $ P_k : V_J \rightarrow V_k $ 
as the orthogonal projections with respect to $(\cdot , \cdot )$ and $a(\cdot , \cdot )$
respectively,
namely, for all $ v \in V_J $ and all  $ w \in V_k $ 
\begin{align}
 (Q_k v , w)  = (v , w ) \,, \quad
 a(P_k v , w)  = a(v , w ) \,.
\end{align}
Note that from this definition it follows that
\begin{align}
  a((I - P_k) v , w) = 0 \quad \mbox{for all} \quad w \in V_k \,.
\end{align}
 
 The multigrid algorithm seeks solutions of the following problem: 
 given $ f \in V_J $, find $ u \in V_J $ such that
 \begin{align}
  a(u,v) = (f, v) \quad \mbox{for all} \quad v \in V_J \,.
 \end{align}
 Before we can present the multigrid algorithm studied in this paper,
 we need to introduce a few operators that will be used in the description of the method.
 
 For $ k = 0,  \dots, J $, define the operators $ A_k : V_k \rightarrow V_k $ as
  \begin{equation}
  (A_k u, v) = a(u , v) \quad \mbox{for all} \quad u,v \in V_k \,.
 \end{equation}
 The operator $A_k$ is SPD with respect to $ (\cdot,\cdot) $  
 as a consequence of the symmetry and positive definiteness of $ a(\cdot,\cdot) $.
 If we set $f_k = Q_k f $, 
 then at level $k$ the problem we want to solve consists in finding $u_k \in V_k$ such that
 \begin{align} \label{lin_sys}
  A_k u_k = f_k.
 \end{align}
 
 The \textit{prolongation} $I_{k-1}^k : V_{k-1} \rightarrow V_k$ 
           and  \textit{restriction} $I^{k-1}_k : V_{k} \rightarrow V_{k-1}$ operators are defined for all $ v \in V_{k-1}$ and all $\, w \in V_k $ by
  \begin{align}
       I_{k-1}^k v  = v , \quad (I^{k-1}_k w , v)  = (w , I_{k-1}^k v) \,.
  \end{align}
  
 We are now ready to present the multigrid algorithm considered in this paper.
 Let $B_k: V_{k} \rightarrow V_{k}$ denote a \textit{smoothing operator}. 
 Associated to $ B_k $ we can define a \textit{smoothing error operator} $S_k : V_{k} \rightarrow V_{k}$
 as $ S_k = I - B_k A_k$,
 whose properties will be discussed later in detail.
 
 For $k=0, \ldots, J$, let  $ MG_k: V_{k} \times V_{k}  \rightarrow V_{k}$ be the \textit{multigrid operators}. 
 The purpose of the operators $  MG_k $ is to yield an approximate solution to  \eqref{lin_sys}.
 They are defined here in a recursive manner.
 \begin{algorithm}[V-cycle multigrid] \label{alg_Vcycle}
 Let $ z_{k}^{(0)}, f_k \in V_{k} $. 
 
 If  $k = 0$, $MG_0(z_{0}^{(0)} , f_0) = A_0^{-1} f_0$ (namely, the exact solution is obtained).

 For $k \geq 1$, $MG_k(z_k^{(0)}, f_k)$ is obtained recursively as follows.
 
 \begin{enumerate}
 \item {\it{Pre-smoothing.}}
  For 1 $\leq i \leq m_k$, let 
 $$z_k^{(i+1)} = z_k^{(i)} + B_k (f_k - A_k z_k^{(i)})\,.$$
 \item {\it{Error Correction.}}
 Let $\bar{f}_k = I^{k-1}_k \left(f_k - A_k z_k^{(m_k)}\right) $,  $q_{k-1} = MG_{k-1}(0, \bar{f}_k)$. Then,
 $$ z_k^{(m_{k} + 1)} = z_k^{(m_{k})} + I_{k-1}^k q_{k-1}  \,. $$
  \item {\it{Post-smoothing.}}
  For $m_k +2 \leq i \leq 2 \, m_k +1 $, let
 $$z_k^{(i+1)} = z_k^{(i)} + B_k (f_k - A_k z_k^{(i)}).$$
 \end{enumerate}

 \end{algorithm}

   Note that the total number of pre-smoothing iterations $m_k$ is assumed to be dependent on the level $k$.
    Also, we are assuming the same number $ m_{k} $ of pre-smoothing and post-smoothing steps at each level.
   We also remark that we consider a symmetric version of the multigrid algorithm as in \cite{bramble1991convergence},
   in the sense that both pre-smoothing and post-smoothing are performed.
   Since we have only one iteration for the error correction step, 
 this algorithm is referred to as V-cycle \cite{braess2007finite}.

\subsection{Convergence analysis}

Here we present a general convergence analysis of the multigrid algorithm \ref{alg_Vcycle}.
We do so by introducing sufficient assumptions on the smoothing error operator $S_k$ 
for the derivation of the convergence results.
It is then clear that the convergence properties of the multigrid algorithm 
are intimately dependent on the smoothing procedures.

Before listing the assumptions, we recall the expressions of the error operators 
associated to $ B_k $ and $MG_k$.
Let $z_k^{(i)}$ be the output of a pre- or post-smoothing iteration at level $k$.
If we denote the associated error as $ e_k^{(i)} = u_k - z_k^{(i)} $, then substituting for $z_k^{(i)}$ we have
\begin{equation}
 e_k^{(i)} =   u_k - z_k^{(i-1)} - B_k (f_k - A_k z_k^{(i-1)}) 
     = S_k \,  e_k^{(i-1)}\,,
\end{equation}
so that the effect of the smoothing step can be described as
\begin{align}
 e_k^{(m_k)} = S_k^{m_k} \, e_k^{(0)} \, .
\end{align}
The \textit{multigrid error operator} $ E_k : V_k \rightarrow V_k$
associated to $MG_k$ is defined recursively as
\begin{equation} \label{mg_error}
 E_0  = 0 \,,\quad  E_k  = S_k^{m_k} [I - (I - E_{k-1}) P_{k-1}] S_k^{m_k} \,.
\end{equation}
Note that $E_0$ is assumed to be zero since we are using a direct solver at level $k=0$.
This means that $B_0 = A_0^{-1}$ and $S_0 = 0$.
Here we summarize the properties of the $ E_k $ operators. 
For a proof see \cite{Brenner} or \cite{bramble1991convergence} for the special case where $z_k^{(0)} = 0$.
\begin{proposition}
Let $ z_k^{(0)} \in V_k $, and let $ u_k $ be the exact solution to $ A_k u_k = f_k $. Then
\begin{displaymath}
   u_k - MG_k( z_k^{(0)},f_k) = E_k \left( u_k  - z_k^{(0)} \right), \qquad k\geq 0.
\end{displaymath}
 Moreover, the $E_k$'s are symmetric positive semidefinite with respect to $a(\cdot , \cdot )$
for $k \geq 0$. 
\end{proposition}

 We now state sufficient hypotheses for multigrid convergence.
 As the expression of the multigrid error operator \eqref{mg_error} suggests,
 once the operators $ P_k $ are given by the differential problem at hand,
  multigrid convergence is affected by the properties of $ S_k $ and by the number of smoothing steps $ m_k $.
 These features are reflected in the following assumptions.
  
\begin{assumption} \label{smooth_SPSD}
 For all $k = 1,  \dots, J \, , \quad S_k$ is a symmetric positive semidefinite operator on $V_k$ 
 with respect to $a(\cdot , \cdot )$.
  This means that for all $v$, $w$ in $V_k$ we have 
  \begin{align}\label{assumpt1}
   a(S_k v , w) = a(v , S_k w) \quad \mbox{and} \quad a(S_kv , v ) \geq 0.
  \end{align}
\end{assumption}

\begin{assumption} \label{smooth_deltak_lt1}
  For all $ k = 1,  \dots, J $ there exists a number $\delta_{k}$ 
    with $0 < \delta_{k} < 1$ such that 
 \begin{align}\label{assumpt2}
  a(S_k v ,v) \leq \delta_{k} \, \, a(v ,v) \quad \mbox{for all} \quad v \in V_k.
 \end{align}
\end{assumption}

\begin{assumption} \label{smooth_psik_nonincr}
  For all $k = 1,  \dots, J$, the finite sequence $ \psi_k = m_k (1-\delta_{k}) $ is non-increasing, 
  where $ m_k $ is the number of smoothing steps per level 
  and $ \delta_k $ is the quantity in Assumption \ref{smooth_deltak_lt1}.
\end{assumption}

\subsubsection{Smoothing and approximation properties}

Assumptions \ref{smooth_SPSD} and \ref{smooth_deltak_lt1} guarantee that the operators $S_k$ 
satisfy certain monotonicity properties 
given by Lemmas \ref{lemma_monot_Sk_1} and \ref{lemma_monot_Sk_2} below.
These properties will lead to the smoothing property of Lemma \ref{lemma_smoothing_prop}.

\begin{lemma} \label{lemma_monot_Sk_1}
Let Assumptions \ref{smooth_SPSD} and \ref{smooth_deltak_lt1} hold.
Let $\alpha$ and $\beta$ be two integers such that $0 \leq \alpha \leq \beta$. Then,
\begin{align}\label{alpha_beta}
 a(S_k^{\beta} v , v ) \leq a(S_k^{\alpha} v , v) \quad \mbox{for all} \quad v \in V_k.
\end{align}
\end{lemma}
\begin{proof}
We will prove it for $\beta=\alpha+1 $ and the result will then follow by induction.
By \eqref{assumpt1}, $S_k$ is positive semidefinite with respect to $ a(\cdot,\cdot)$,
therefore also $S_k^{\alpha}$ is. 
Then there is a unique positive semidefinite square root operator $S_k^{\frac{\alpha} {2}}$, see \cite{horn2012matrix}.
By the symmetry of $S_k$ it follows
that $S_k^{\frac{\alpha} {2}}$ is symmetric as well.
Considering also \eqref{assumpt2},
we have that for $v \in V_k$
\begin{align*}
 a(S_k^{\alpha+1}v ,v) 
 &= a(S_k^{\alpha}S_k v , v) = a(S_k^{\frac{\alpha}{2}}S_k^{\frac{\alpha}{2}}S_k v , v)  \\
 &= a(S_k^{\frac{\alpha}{2}}S_k v , S_k^{\frac{\alpha}{2}}v) = a(S_k S_k^{\frac{\alpha}{2}}v , S_k^{\frac{\alpha}{2}}v)  \\
 &\leq a(S_k^{\frac{\alpha}{2}}v , S_k^{\frac{\alpha}{2}}v) = 
  a(S_k^{\alpha} v , v)\,. 
\end{align*}
\end{proof}

Using the previous lemma, we can prove the next result.

\begin{lemma}  \label{lemma_monot_Sk_2}
 Let Assumptions \ref{smooth_SPSD} and \ref{smooth_deltak_lt1} hold.
 Let $\alpha$ and $\beta$ be two integers such that $0 \leq \alpha \leq \beta$. Then,
\begin{align}
 a((I-S_k)S_k^{\beta} v , v ) \leq a((I-S_k)S_k^{\alpha} v , v) \quad \mbox{for all} \quad v \in V_k.
\end{align}
\end{lemma}
\begin{proof}
As we did before, we will prove it for $\beta=\alpha+1 $ and the result will follow by induction.
\begin{align*}
 a((I-S_k)S_k^{\alpha+1} v , v )  
 & = a((I-S_k)S_k^{\alpha+1} v , (I -S_k + S_k)v ) \\
 & = a((I-S_k)S_k^{\alpha+1} v , (I -S_k)v) + a((I-S_k)S_k^{\alpha+1} v , S_kv) \\
 & = a(S_k^{\alpha+1}(I-S_k) v , (I -S_k)v) + a((I-S_k)S_k^{\alpha+1} v , S_kv) \\
 & \leq a(S_k^{\alpha}(I-S_k) v , (I -S_k)v) + a((I-S_k)S_k^{\alpha+1} v , S_kv) \qquad \mbox{(by \eqref{alpha_beta} )}\\
 = \, \, a(S_k^{\alpha}(I-S_k) v , v) &- a(S_k^{\alpha}(I-S_k) v , S_k v)  + a((I-S_k)S_k^{\alpha+1} v , S_kv)  \\
 = \, \, a(S_k^{\alpha}(I-S_k) v , v) &- a(S_k^{\alpha+1}(I-S_k) v , v)  + a(S_k^{\alpha+1}(I-S_k) v , S_kv) \\
 &= a(S_k^{\alpha}(I-S_k) v , v) - a(S_k^{\alpha+1}(I-S_k) v , (I-S_k)v) \\
 &\leq a(S_k^{\alpha}(I-S_k) v , v) \qquad \mbox{($S_k^{\alpha+1}$ is positive semidefinite, see \cite{horn2012matrix})} \\
 &= a((I-S_k) S_k^{\alpha}v , v) \,.
\end{align*}
\end{proof}

Now we are ready to prove the smoothing property of the operator $S_k$.

\begin{lemma}[Smoothing property] \label{lemma_smoothing_prop}
 Let Assumptions \ref{smooth_SPSD} and \ref{smooth_deltak_lt1} hold.
Let $ v \in V_k $. Then,
\begin{align}
 a((I - S_k)S_k^{2m_k}v ,v ) \leq \dfrac{1}{2m_k} \, a((I - S_k^{2m_k})v , v).
\end{align}
\end{lemma}
\begin{proof}
By Lemma \ref{lemma_monot_Sk_2} we get
\begin{align*}
& (2m_k)a((I-S_k)S_k^{2m_k}v , v) =  \underbrace{a((I-S_k)S_k^{2m_k}v , v)  + \dots + a((I-S_k)S_k^{2m_k}v , v)}_{\text{$2 m_k$ times}}   \\
& \leq a((I-S_k)v , v) + a((I-S_k)S_kv , v) + \dots + a((I-S_k)S_k^{2m_k - 1}v , v) \\
& = a( (I - S_k + S_k - S_k^2 + \dots + S_k^{2m_k-1} - S_k^{2m_k}) v , v) \\
&  = a((I-S_k^{2m_k})v , v) \,.
\end{align*}
\end{proof}


Before we can show a bound on the error $E_k$, we need to establish the approximation property.
\begin{lemma} [Approximation property] \label{lemma_1st_prelim_res} 
 Let Assumptions \ref{smooth_SPSD} and \ref{smooth_deltak_lt1} hold, and let $w \in V_k$.
 Then,
\begin{align}
 a((I - P_{k-1})w , w) \leq \Big( \, \dfrac{1}{1-\delta_{k}} \, \Big)a((I-S_k)w, w) \,.
\end{align}
\end{lemma}
\begin{proof}
Let $ y = (I - P_{k-1}) w $. Note that from the definition of $P_{k-1}$
and from the nestedness of the spaces \eqref{mg_spaces} 
we have that 
\begin{align} \label{UorthPW}
a(y , P_{k-1}w) = 0  \,.
\end{align}
By Assumption \ref{smooth_deltak_lt1} we have
\begin{align*}
 a((I-S_k)w , w) &\geq (1 - \delta_{k}) \, \,  a(w , w)  \\
 &= (1 - \delta_{k}) \, \, a(y + P_{k-1}w , y + P_{k-1}w) \\
 &= (1 - \delta_{k}) \, \, ( a(y , y) + a(P_{k-1}w , P_{k-1}w)) \quad \text{(by \eqref{UorthPW})} \\
 &\geq (1 - \delta_{k}) \, \, a((I-P_{k-1})w , (I-P_{k-1})w) \\
  &= (1 - \delta_{k}) \, \, a((I-P_{k-1})w , w) \qquad \qquad \quad \, \, \mbox{(by \eqref{UorthPW})} \,.
\end{align*}
\end{proof}
Notice that in this setting the approximation property given by Lemma \ref{lemma_1st_prelim_res} 
is dependent on the smoothing property. 
 This is in contrast with other analyses of multigrid methods
 in which the smoothing and the approximation properties are derived independently \cite{Brenner}.
As a consequence of the approximation property, we have the following result.
\begin{lemma} \label{lemma_2nd_prelim_res}  
 Let Assumptions \ref{smooth_SPSD} and \ref{smooth_deltak_lt1} hold.
 Let $v \in V_k$. Then,
\begin{equation}
 a((I-P_{k-1})S_k^{m_k}v , (I-P_{k-1})S_k^{m_k}v)  
   \leq   \Big( \, \dfrac{1}{1-\delta_{k}} \, \Big)\dfrac{1}{2m_k} \, \, a((I -S_k^{2m_k})v , v) \,. 
\end{equation}
\end{lemma}
\begin{proof}
  We have
\begin{align*}\label{2ndprelim}
 &a((I-P_{k-1})S_k^{m_k}v , (I-P_{k-1})S_k^{m_k}v) \\
 \qquad & = a((I-P_{k-1})S_k^{m_k}v , S_k^{m_k}v) \qquad \qquad \mbox{(from the definition of $P_{k-1}$)}  \\
 \qquad & \leq \Big( \, \dfrac{1}{1-\delta_{k}} \, \Big) \, a((I-S_{k})S_k^{m_k}v , S_k^{m_k}v) \quad \mbox{(by Lemma \ref{lemma_1st_prelim_res})} \\
 \qquad & = \Big( \, \dfrac{1}{1-\delta_{k}} \, \Big) \, a((I-S_{k})S_k^{2m_k}v , v)  \\
 \qquad & \leq  \Big( \, \dfrac{1}{1-\delta_{k}} \, \Big) \dfrac{1}{2 m_k} a ((I - S_k ^{2 m_k})v , v) \quad \mbox{(by Lemma \ref{lemma_smoothing_prop})} \,. 
\end{align*}
\end{proof}

\subsubsection{Error bound}

We are now in a position to obtain a bound on the multigrid error operator $E_J$ that gives convergence. 
\begin{theorem} \label{mg_conv}
 Let Assumptions \ref{smooth_SPSD}, \ref{smooth_deltak_lt1} and \ref{smooth_psik_nonincr} hold.
 For $ k = 0, 1, \ldots, J $ let 
 \begin{equation} \label{gamma_k}
\gamma_{k} = \dfrac{1}{1 + 2 \, m_k (1 - \delta_{k})}. 
 \end{equation}
 Then, if $v \in V_J$,
\begin{align}
 a(E_J v , v) \leq \, \gamma_{J} \, a(v ,v).
\end{align}
\end{theorem}
\begin{proof}
By Assumption \ref{smooth_psik_nonincr}, we have
\begin{equation}\label{gamma}
\gamma_{k-1} \leq \gamma_k \,.
\end{equation}
The proof will be done by induction as in \cite{Brenner}.
For $k=0$, $E_0 = 0$ so the result is obvious. By induction assume that
$$a(E_{J-1} v , v) \leq \, \gamma_{J-1} \, a(v ,v) \qquad \forall v \in V_{J-1}\,.$$
Now consider $v \in V_J$, then
\begin{align*}
a(E_J v , v) &= a(S_J^{m_J} v , S_J^{m_J} v) - a(P_{J-1}S_J^{m_J}v , P_{J-1}S_J^{m_J}v)  \\
 & \quad + \, a(E_{J-1}P_{J-1}S_{J}^{m_J} v, P_{J-1}S_J^{m_J}v)  \\
&= a((I - P_{J-1})S_J^{m_J}v , (I - P_{J-1})S_J^{m_J}v)  \\
 & \quad + \, a(E_{J-1}P_{J-1}S_{J}^{m_J} v, P_{J-1}S_J^{m_J}v) \qquad \qquad \, \,  \mbox{(by definition of $P_{J-1}$)}   \\
& \leq  a((I - P_{J-1})S_J^{m_J}v , (I - P_{J-1})S_J^{m_J}v)  \\
& \quad  +  \, \gamma_{J-1} \, a(P_{J-1}S_{J}^{m_J} v, P_{J-1}S_J^{m_J}v) \quad \mbox{(by the induction assumption)}  \\
&\leq a((I - P_{J-1})S_J^{m_J}v , (I - P_{J-1})S_J^{m_J}v)  \\
& \quad  +  \, \gamma_{J} \, a(P_{J-1}S_{J}^{m_J} v, P_{J-1}S_J^{m_J}v) \qquad \qquad \qquad \qquad \qquad \quad \mbox{(by \eqref{gamma})}  \\
&=  (1-\gamma_J) \, a((I - P_{J-1})S_J^{m_J}v , (I - P_{J-1})S_J^{m_J}v)  \\
& \quad + \, \gamma_J \, a((I - P_{J-1})S_J^{m_J}v , (I - P_{J-1})S_J^{m_J}v)  \\
& \quad  +  \, \gamma_{J} \, a(P_{J-1}S_{J}^{m_J} v, P_{J-1}S_J^{m_J}v)  \\
&=  \, (1-\gamma_J) \, a((I - P_{J-1})S_J^{m_J}v , (I - P_{J-1})S_J^{m_J}v)  
\\
& \quad  +  \, \gamma_{J} \, a(S_{J}^{m_J} v, S_J^{m_J}v) \qquad \qquad \qquad \qquad \qquad \mbox{(by definition of $P_{J-1}$)}  \\
 \\
&\leq  \Big( \, \dfrac{1}{2 m_J \, (1-\delta_{J}) } \, \Big) \, (1-\gamma_J) \, a((I - S_J^{2m_J})v ,v)  \\
 \\
& \quad  +  \, \gamma_{J} \, a(S_{J}^{m_J} v, S_J^{m_J}v)\qquad \qquad \qquad \mbox{(by Lemma \ref{lemma_2nd_prelim_res})}   \\ 
 \\
&= \gamma_J \, a((I - S_J^{2m_J})v ,v) + \gamma_{J} \, a(S_{J}^{m_J} v, S_J^{m_J}v)  \\
&= \gamma_J \, a(v,v) - \gamma_J \, a(S_J^{m_J}v , S_J^{m_J}v) + \gamma_J \, a(S_J^{m_J}v, S_J^{m_J}v)  \\
&= \gamma_J \, a(v,v)  
\end{align*}
\end{proof}

It is evident that the convergence of the multigrid algorithm
is dependent on Assumption \ref{smooth_psik_nonincr},
which lies 
both on the number of smoothing steps $ m_k $ 
 and on the constants $ \delta_k $ of the smoothing error operator, jointly.
 No other parameters affect the convergence rate.
 Since the behavior of $ \delta_k$ is determined by the choice of $ S_k $,
 different choices on $m_k$ can be taken subsequently so that \eqref{gamma} holds.
 For instance, if the $ S_k $ are such that $ \delta_k $ is non-decreasing, 
 then it is sufficient to take $ m_k $ to be non-increasing. 
 In fact, if $ \delta_k \geq \delta_{k-1} $ and $ m_k \leq m_{k-1} $,
 Assumption \ref{smooth_psik_nonincr} holds since 
 \begin{equation}
    m_k (1-\delta_{k}) \leq m_{k-1} (1-\delta_{k-1})\,.
 \end{equation}
 Although sufficient,  a non-increasing $ m_k $ is not necessary.
Also, observe that a particular case of non-increasing $ m_k$ is 
 $ m_1 = m_2 = \dots = m_J = m $.
In this case it is directly visible how an increasing $ m $ can lead to better convergence rates,
as well as an increasing number of multilevel spaces $J$ can lead to worse convergence.
While this last situation was shown in \cite{bramble1991convergence},
to the best of our knowledge the first feature 
was not shown in similar multigrid frameworks without, or with minimal regularity assumptions 
(see, e.g., \cite{bramble1991convergence, chen2012optimal, brenner2002convergence}).

In Section \ref{sec_smoother}  
a characterization of $ \delta_k $ will be given
when the smoothing process is chosen to be a successive subspace correction algorithm. 
From this characterization, proper choices of the number of smoothing steps $ m_k $ 
can lead to convergence and, in addition, to optimal (i.e., with a value of $ \gamma_J $ that is independent of $J$) 
multigrid error bounds.
 Since $ m_k \in \mathbb{N} $ and $ \delta_k \in \mathbb{R} $,
 our analysis suggests that the determination of an optimal multigrid error bound 
 may take place by a proper choice of $ m_k $
 if and only if the quantity $(1-\delta_k)$ is inversely proportional  
 to an integer-valued function of $ k $.
 The achievement of an optimal convergence bound 
 seems otherwise impossible, unless a radically new setup of the multigrid algorithm is formulated
 that is oriented to that purpose.


\section{Successive Subspace Correction (SSC) algorithms}
\label{sec_smoother}


Now we describe the Successive Subspace Correction (SSC) algorithm.
The SSC algorithm is an iterative method to approximate the solution of SPD linear systems \cite{xu1992iterative}.
\giachi{We link the multigrid convergence theory of the previous section with the SSC theory by using
smoothers of subspace correction type for the multigrid algorithm \ref{alg_Vcycle}.}
The SSC algorithm yields an approximate solution to \eqref{lin_sys} 
and is based on a decomposition of the finite-dimensional space as an algebraic sum of subspaces.
In the multigrid algorithm presented above, smoothing is performed at each level $k = 1, ... , J$, 
therefore we decompose each $V_k$ using subspaces $ V_k^i \subset V_k $ such that
\begin{align} \label{decomp_SSC}
 V_k = \sum_{i=0}^{p_k} V_k^i = \Big \{ v \, \, | \, \,  v = \sum_{i=0}^{p_k} v_k^i \, , \, v_k^i \in V_k^i \Big \} \,. 
 \end{align}
Notice that the number of subspaces $p_k$ is in general different for each level.
In order to present the algorithm, we first define  for all $i$,
 with $ u \in V_k $ and $ u_k^i , v_k^i \in V_k^i $, the operators
\begin{align*}
& Q_k^i : V_k \rightarrow V_k^i\,, \quad P_k^i: V_k \rightarrow V_k^i \,, \quad A_k^i : V_k^i \rightarrow V_k^i  \,,  \\
& (Q_k^i u , v_k^i) = (u , v_k^i), \quad a(P_k^i u , v_k^i) = a(u , v_k^i) \,, \\
& (A_k^i u_k^i , v_k^i) = (A_k u_k^i , v_k^i) \,.
\end{align*}
It follows from its definition that $A_k^i$ is an SPD operator.
Moreover, as a consequence of the above definitions we have that 
\begin{align}\label{apqak}
 A_k^i P_k^i = Q_k^i A_k\,.
\end{align}
Hence if $u_k$ is the exact solution of \eqref{lin_sys}, then $P_k^i u_k = u_k^i$ will be the solution of 
\begin{align}\label{discprobsubsp}
 A_k^i u_k^i = f_k^i,
\end{align}
where $f_k^i = Q_k^i f_k$.
Equation \eqref{discprobsubsp} is in general solved approximately,
therefore we introduce for all $i$ the operators
$$ R_k^i : V_k^i \rightarrow V_k^i\,, \quad T_k^i : V_k \rightarrow V^i_k \,, \quad T_k^i := R_k^i Q_k^i A_k = R_k^i A_k^i P_k^i \,. $$ 
The operators $R_k^i$ act as approximate inverses of $A_k^i$. 
If $R_k^i$ is taken to be an exact solver then $T_k^i = P_k^i$. 
When no confusion arises, we drop the subscript $k$ for $p_k$ 
as well as for the operators $Q_k^i$, $P_k^i$, $A_k^i$, $R_k^i$ and $T^i_k$.
We now define the SSC algorithm.
\begin{algorithm}[Successive Subspace Correction Algorithm.] \label{alg_ssc}
Let $z^0 \in V_k$ be given.
Then $z^{\alpha+1}$ is obtained in $p+1$ substeps starting from $z^{\alpha}$ by
\begin{align}
 z^{\alpha + 1 - \frac{i}{p+1}} = z^{\alpha + 1 - \frac{i+1}{p+1}} + R^i Q^i (f_k -A_k z^{\alpha+ 1 -\frac{i+1}{p+1}}) ,
\end{align}
for $i = p , \dots , 0$.
\end{algorithm}
The error operator associated to this algorithm is denoted as $\widehat{E}_p $.
If $u_k$ is the exact solution of \eqref{lin_sys}, then for $i = p , \dots , 0$ we have
$$ (u_k - z^{\alpha + \frac{p+1-i}{p+1}}) = (I - T^i) (u_k - z^{\alpha+ \frac{p-i}{p+1}}) \,.$$
This yields
\begin{align}
 (z_k - z^{\alpha + 1}) & = \widehat{E}_p (z_k - z^{\alpha})\,, \\
  \widehat{E}_p  & = (I - T^0) ( I - T^1) \cdots (I - T^p) \,.
 \end{align}
 The symmetric version of the SSC algorithm is given here.
\begin{algorithm}[Symmetric Successive Subspace Correction Algorithm.] \label{alg_sym_ssc}
Let $z^0 \in V_k$ be given.
First, $z^{\alpha + \frac{1}{2}}$ is obtained in $p+1$ substeps starting from $z^{\alpha}$ by
\begin{align}
 z^{\alpha + \frac{1}{2} - \frac{i}{2(p+1)}} = z^{\alpha +\frac{1}{2} - \frac{i+1}{2(p+1)}} + R^iQ^i(f_k - A_kz^{\alpha  +\frac{1}{2} - \frac{i+1}{2(p+1)} }) ,
\end{align}
for $i = p , \dots , 0.$
Then, $z^{\alpha+1}$ is obtained in $p+1$ substeps starting from $z^{\alpha+\frac{1}{2}}$ by
\begin{align}
 z^{\alpha + \frac{1}{2} + \frac{(i+1)}{2(p+1)}} = z^{\alpha + \frac{1}{2} + \frac{i}{2(p+1)}} + R^iQ^i(f_k -A_kz^{\alpha + \frac{1}{2} + \frac{i}{2(p+1)} }) ,
\end{align}
for $i = 0, \dots , p.$
\end{algorithm}
The error operator of this algorithm is denoted as $\widehat{E}_p^s $ and is given by
\begin{align}
 \widehat{E}^s_p = (I - T^p) \cdots ( I - T^1)(I - T^0)^2( I - T^1) \cdots (I - T^p) \,.
\end{align}
Because of the symmetry requirements for the smoother in Assumption \ref{smooth_SPSD},
we will fit the smoother within the symmetric SSC framework.

\subsection{Convergence analysis}

We recall the main convergence result about the SSC algorithm, 
whose proof can be found in \cite{xu1992iterative}.
First, we introduce sufficient assumptions. 
\begin{assumption}[Bound on $ w_1 $] \label{RiSPD_w1lt2}
 The  operators $R^i$ are SPD with respect to $ (\cdot,\cdot) $ and satisfy $ w_1 < 2 $, 
where $ w_1 = \max\limits_{i=0, \dots ,p} \rho(R^i A^i) $,
$ \rho(R^i A^i) $ being the spectral radius of $ R^i A^i $
and $p$ being the number of subspaces in the decomposition \eqref{decomp_SSC}.
\end{assumption}
\begin{assumption}[Existence of $ K_0 $] \label{exist_K0}
 There exists  $K_0$ such that for any $ v \in V_k $ there exists a decomposition
$ v = \sum \limits_{i=0}^p v_i $, with the property
\begin{align} \label{Ri_inv}
 \sum_{i=0}^p ( (R^i)^{-1} v_i , v_i) \leq K_0 \, (A_k v , v) \,.
\end{align}
\end{assumption}
\begin{assumption}[Existence of $ K_1 $] \label{exist_K1}
 Given the same $p$ as in Assumption \ref{exist_K0}, 
 there exists  $ K_1 $ such that 
 for any $S \subset \{0,1, \dots , p\} \, \times \, \{0,1, \dots , p\}$ and $ u_i , v_i \in V_k $ for 
$ i = 0,1, \dots , p $ we have
\begin{align} \label{ineq_K1}
 \sum_{(i,j)\in S} | a(T^i u_i , T^j v_j ) |
 \leq K_1 \, \Big( \sum_{i=0}^p a(T^i u_i , u_i ) \Big)^{\frac{1}{2}} \, \Big( \sum_{j=0}^p a(T^j v_j , v_j ) \Big)^{\frac{1}{2}} \,.
\end{align}
\end{assumption}
%
Notice that all assumptions are related to the choice of the operators $ R^i $.
Assumption \ref{exist_K1} involves only functions in $ V_k $, without using the decomposition in Assumption \ref{exist_K0}.
\giachi{We remark that the absolute value in Equation \eqref{ineq_K1} is sufficient but not necessary for convergence, see \cite{xu1992iterative}.}
Due to Assumption \ref{RiSPD_w1lt2}, $ R^i $ is invertible so that \eqref{Ri_inv} is well-defined.
Also, recall the following property.
\begin{lemma} \label{lemma_Ti}
Let Assumption \ref{RiSPD_w1lt2} hold.
The operator $ T^i $ is symmetric and positive semi-definite with respect to $ a(\cdot , \cdot ) $.
\end{lemma}
\begin{proof}
Let $ u $ and $ v $ be in $ V_k $. 
Using the fact that
$ A_k $ is symmetric with respect to $ (\cdot,\cdot) $ in $ V_k $,
$ R_k^i $ is symmetric with respect to $ (\cdot,\cdot) $ in $ V_k^i $,
together with the definition of $ Q_k^i $, we have
\begin{align*}
 a(T_k^i u, v) & = a(R_k^i\,Q_k^i\,A_k\,u, v) = (R_k^i\,Q_k^i\,A_k\,u ,  A_k\,v) \\
 & = (R_k^i \, Q_k^i \, A_k\,u, Q_k^i\,A_k\,v) = (Q_k^i\,A_k\,u, R_k^i \,Q_k^i\,A_k\,v) \\
 & = (Q_k^i\,A_k\,u, T_k^i\,v) = (A_k\,u, T_k^i\,v) = a(u, T_k^i\,v).
\end{align*}
This shows $T_k^i$ is symmetric with respect to $a(\cdot,\cdot)$.
To see that it is also positive semi-definite, we use the same properties as above and get
\begin{align*}
 a(T_k^i u, u) = (R_k^i Q_k^i\,A_k\,u , A_k\,u)  =   (R_k^i Q_k^i\,A_k\,u, \,Q_k^i\,A_k\,u).
\end{align*}
Since $R_k^i$ is SPD with respect to $(\cdot,\cdot)$ by Assumption \ref{RiSPD_w1lt2}, the result follows.
\end{proof}
By the symmetry of $ T^i $ with respect to $ a(\cdot,\cdot)$,
$ I - T^i $ is symmetric with respect to  $ a(\cdot,\cdot)$.
Hence $ \widehat{E}_p^* = (I - T^p) \cdots ( I - T^1)(I - T^0) $ is
the adjoint of $ \widehat{E}_p $ with respect to $ a(\cdot,\cdot)$,
so that 
\begin{equation} \label{sym_nonsym}
\widehat{E}_p^s = \widehat{E}_p^* \widehat{E}_p.
\end{equation}
Thus, we have 
\begin{align}\label{energynormEhat}
 a(\widehat{E}_p^s v , v) = || \widehat{E}_p v||_E^2 \quad \forall  v \in V_k. 
\end{align}

We now state the convergence result.
\begin{theorem}
Let Assumptions \ref{RiSPD_w1lt2}, \ref{exist_K0} and \ref{exist_K1}  hold.
Then, we have
\begin{align}\label{inequalityEhat}
|| \widehat{E}_p ||_E^2 \leq 1 - \dfrac{2 - w_1}{K_0 \, (1 + K_1)^2} \,,
 \end{align}
where $w_1$ was defined in Assumption \ref{RiSPD_w1lt2} and $K_0$ and $K_1$
are constants related to the ones in Assumptions \ref{exist_K0} and \ref{exist_K1}.
\end{theorem}
\begin{proof}
 See \cite{xu1992iterative} for a proof.
 Let us point out that the quantity in the right-hand side of \eqref{inequalityEhat} 
 has a nonzero denominator, is larger than 0 and less than 1.
 In fact, notice that Assumption \ref{RiSPD_w1lt2} implies that $ (R^i)^{-1} $ is also SPD with respect to $ (\cdot,\cdot) $,
 which, together with the fact that $ A $ is SPD with respect to $ (\cdot,\cdot) $,
 implies $ K_0 > 0$.
 Then,  
 $ \dfrac{2 - w_1}{K_0 \, (1 + K_1)^2} $ is well-defined and 
 by Assumption \ref{RiSPD_w1lt2} we have
 $ \dfrac{2 - w_1}{K_0 \, (1 + K_1)^2} > 0 $.
 Finally, the constant $ K_0 $ can be majorized by another constant 
 such that \eqref{Ri_inv} still holds and $ \dfrac{2 - w_1}{K_0 \, (1 + K_1)^2} < 1 $.
\end{proof}

The convergence of the symmetric version of the SSC algorithm
is then a direct consequence of \eqref{sym_nonsym}.

\begin{remark}
 As we dropped the subscript $k$ to make the notation more readable, we point out
 that the quantities $w_1$, $K_0$ and $K_1$ in general depend on $k$.
 Moreover, also the quantity $p$ in Algorithms \ref{alg_ssc} or \ref{alg_sym_ssc} 
 can be chosen differently for different multigrid levels.
\end{remark}

\subsection{Sufficient conditions for multigrid convergence with smoothers of SSC type}

Our intent is to fit the properties of the SSC smoother to the sufficient conditions needed 
for the convergence of the multigrid algorithm \ref{alg_Vcycle}.
Choosing the symmetric SSC iteration as the smoother for our multigrid algorithm, we then have
\begin{align}
 S_k = \widehat{E}_{p_k}^s \,.
\end{align}
Our purpose is to consider a set of choices of $ V_k $ in the multigrid algorithm 
and of smoothers $ S_k = \widehat{E}_{p_k}^s $ 
for which Assumptions \ref{smooth_SPSD}, \ref{smooth_deltak_lt1} and \ref{smooth_psik_nonincr}
are satisfied. 
First, the existence of a suitable error operator $\widehat{E}_{p_k}^s $ with norm less than one is given
by the fulfillment of Assumptions \ref{RiSPD_w1lt2}, \ref{exist_K0} and \ref{exist_K1}.
Once the existence of this operator is granted, Assumptions \ref{smooth_SPSD} and \ref{smooth_deltak_lt1} are true.

In fact, Assumption  \ref{smooth_SPSD} is a consequence of the fact that $ \widehat{E}_{p_k}^s $
is SPD with respect to $ a(\cdot,\cdot)$.
Concerning Assumption \ref{smooth_deltak_lt1}, 
we set $ \delta_{k} $ as
\begin{align} \label{deltak_choice}
 \delta_{k} = 1 - \dfrac{2 - w_1}{K_0 \, (1 + K_1)^2}\,.
\end{align}
It then follows by \eqref{energynormEhat} and \eqref{inequalityEhat} that 
Assumption \ref{smooth_deltak_lt1} holds.

The only assumption that remains to be checked is Assumption \ref{smooth_psik_nonincr},
which again depends on all the Assumptions \ref{RiSPD_w1lt2}, \ref{exist_K0} and \ref{exist_K1}.
This is due to the fact that $ \delta_k $ in \eqref{deltak_choice} is determined by $w_1 $, $K_0 $ and $K_1$, 
all of which in general depend on $k$. 
Different definitions of $ V_k $ and $ S_k$ correspond to multigrid algorithms 
on different spaces with different subspace correction smoothing schemes.
Some examples will be described in Section \ref{sec_app}.
In these, a verification of Assumptions \ref{RiSPD_w1lt2}, \ref{exist_K0}, \ref{exist_K1}
is provided, 
along with a characterization of the constants $w_1 $, $K_0 $ and $K_1$ in terms of $k$. 
This characterization leads to the identification of the conditions
for Assumption \ref{smooth_psik_nonincr} to hold.
The conditions for the optimality of the multigrid error bound are also determined.


\section{Refinement applications with subspace correction smoothing schemes}
\label{sec_app}

In this section, we apply the multigrid algorithm with SSC-type smoother described earlier,
to applications involving uniform and local refinement and also domain decomposition smoothing.
Multiplicative domain decomposition algorithms can in fact be seen as instances of SSC algorithms \cite{xu1992iterative}.
In the following, we introduce the model problem and its finite element discretization.
The applications of the theory that we consider here differ 
in the definition of $ V_k $, 
in its decomposition into appropriate subspaces $ V^i_k $ 
and in the choice of the operators $ R^i_k $.
We first address a case of uniform refinement with exact subsolvers.
Then, we present two local refinement applications 
that deal with two possible ways of enforcing continuity, 
corresponding to appropriate choices of the subspace decomposition of the multigrid spaces $V_k$. 

\subsection{Model problem and finite element discretization}

To fix the ideas,
let $\Omega$ be a polygonal subset of $\mathbb{R}^n$,
let $ \Theta = (\theta_{ij}) $ be a symmetric matrix
and consider the elliptic boundary value problem
\begin{align*}
 - \sum_{i=1}^{n} \sum_{j=1}^{n}\, \dfrac{\partial}{\partial x_i} \left( \theta_{ij} \dfrac{\partial u}{\partial x_j} \right) = f\quad &\mbox{in} \quad \Omega  \\
 u = 0 \qquad &\mbox{on}  \quad \partial\Omega, 
\end{align*}
then $u$ is a weak solution of the above problem if and only if
\begin{align}
 a(u,v) = (f,v) \quad \mbox{for all} \quad v \in H_0^1(\Omega)
\end{align}
where $ (\cdot, \cdot) $ denotes the $ L^2(\Omega) $ inner product and 
\begin{align}
 a(u,v) = - \sum_{i=1}^{n} \sum_{j=1}^{n} \int_{\Omega}  \theta_{ij} \dfrac{\partial u}{\partial x_j} \dfrac{\partial v}{\partial x_i} \, dx.
\end{align}
We assume that there exist $C_{\Theta,1}$ and $C_{\Theta,2}$ depending on $ \Theta $
for which
\begin{align} \label{bil_equiv_H1} 
 C_{\Theta,1} \| u \|^2_{H_1} \leq a(u,u) \leq C_{\Theta,2} \| u \|^2_{H_1}, \qquad u \in V_k \subset H_0^1(\Omega).
\end{align}
This means that both $(\cdot , \cdot)$ and $a(\cdot, \cdot)$ are SPD bilinear forms on $ V_k $,   
as needed in the previous convergence theory.
Moreover, since the trace of $ V_k $ is zero on the boundary of $\Omega$,
we have that $ a(\cdot , \cdot) $ is also equivalent to $|\cdot|_{H^1(\Omega)}$ on $V_k$
due to the Poincar\'e inequality.

Let $\mathcal{P}_1$ be the space of linear polynomials, 
then the multigrid spaces $V_k$ in \eqref{mg_spaces} 
will be considered to be the finite element spaces 
of continuous piecewise-linear functions
built on triangulations $\mathcal{T}_k$ of $\Omega$,
\begin{equation} \label{Vk}
 V_k = \{ v \in H_0^1(\Omega) : v|_{\tau} \in \mathcal{P}_1, \, \,  \forall \tau \in \mathcal{T}_k \} \quad k=0, \dots, J.
\end{equation}
Such triangulations will be defined 
by using either uniform or local midpoint refinement.
In the case where such refinement procedure is performed only on a subdomain of $\Omega$ (local refinement),
hanging nodes will be introduced in the mesh and the triangulation will be referred to as {\it{irregular}} \giachi{(or {\it non-conforming})}.
Hanging nodes (also called slave nodes by some authors) are vertices of some element $\tau_1 \in \mathcal{T}_k$ 
that lie on the interior of an edge of some
other element $\tau_2 \in \mathcal{T}_k$ \giachi{without being a node for $\tau_2$}. 
A more formal description of hanging nodes can be found in \cite{carstensen2009hanging, heuveline2007inf, fries2011hanging, di2016easy}. 
Continuity constraints can be added in the definition of the finite element spaces 
to make sure that no additional degrees of freedom are introduced for the hanging nodes. 
Therefore for all $ k=0, \dots, J$, $V_k $ has a nodal basis that consists of \giachi{functions associated to} 
all vertices of $\mathcal{T}_k$ excluding the hanging nodes.
\giachi{In practice, a possible way to obtain a continuous nodal basis is given in \cite{fries2011hanging},
where shape functions of elements with hanging nodes in the element corners are modified}.
In \cite{fries2011hanging}, the support of a basis function associated to a regular node $n$ is the union of elements that share this node or the
potential hanging nodes on the edges that have node $n$.
\giachi{We point out that the local refinement applications covered by our theory allow the presence of edges
with an arbitrary number of hanging nodes. Usually, only $1$-irregular meshes are considered \cite{zhao2010adaptive}, namely meshes where at most one hanging node
is allowed on any edge of the triangulation.}

\subsection{Uniform refinement:
overlapping non-nested subdomains, 
subproblems on regular grids 
and exact subsolvers}

We first describe a case of uniform refinement by defining the triangulations and the corresponding subdomains on each of them.
\begin{definition}[Triangulations $\mathcal{T}_k$] \label{def_triang_1}
Let $\mathcal{T}_0$ be a quasi-uniform coarse triangulation of $\Omega$ of size $h_0 \in (0,1]$.
Assume $\mathcal{T}_{k-1}$ has been obtained,
then $\mathcal{T}_k$ is derived from $\mathcal{T}_{k-1}$ by means of midpoint refinement. It follows that the size
$h_k$ of $\mathcal{T}_k$ will be $h_k = 2^{-k} h_0$ and that
$$ \mathcal{T}_0 \subset \mathcal{T}_1 \subset \cdots \mathcal{T}_k,$$
in the sense that any $ \tau \in \mathcal{T}_k$ can be written as the union of elements in $\mathcal{T}_{k+1}$ \cite{Ciarlet:2002:FEM:581834}.
\end{definition}
\begin{definition}[Subdomains $ \widehat{\Omega}^i_k $] \label{def_subd_1}
 Let $\{\Omega^i_k\}_{i=1}^{p_k}$ be a collection of non-overlapping open subdomains of $\Omega$ \giachi{whose boundaries align with the mesh triangulation $\mathcal{T}_k$}, such that
$\overline{\Omega} = \bigcup \limits_{i=1}^{p_k} \overline{\Omega^i_k}$.
For $i=1, \ldots, p_k$, let $\widehat{\Omega}^i_k$  be overlapping subsets of $\Omega$ 
whose boundaries \giachi{still} align with the triangulation and are defined by 
\giachi{\begin{align}
 \widehat{\Omega}^i_k = \{ x \in \Omega \, | \, \mbox{dist}(x, \Omega^i_k) \leq  h_0 \} \, \, .
\end{align}}
\end{definition}
Notice that the number of subdomains $p_k$ varies with the level.
An example of a subdomain described in the above definition is shown in Figure \ref{uniform_case}.
\begin{figure}[h]
\centering
\includegraphics[scale=0.9]{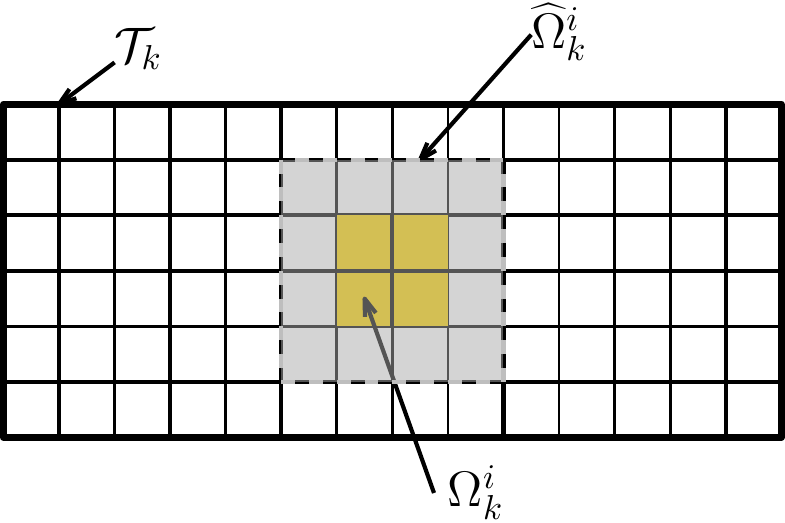}
\caption{ Example of a subdomains involved in the uniform refinement application.}\label{uniform_case}
\end{figure}

For this application, 
the multigrid spaces $V_k$ in \eqref{Vk}, 
the subspaces $ V_k^i $ and 
the subsolvers $ R_k^i $ are defined as follows.
\begin{definition} \label{def_VR_1}
Given the triangulations $\mathcal{T}_k $ in Definition \ref{def_triang_1}
and the overlapping subdomains $\widehat{\Omega}^i_k$ in Definition \ref{def_subd_1}, 
we set for $ k = 0, \ldots, J $ and for $i=0, \ldots, p_k $
\begin{align} \label{spaces_smoother_app1}
\begin{cases}
 V_k \text{ in \eqref{Vk}, built on $\mathcal{T}_k$ as in Definition \ref{def_triang_1}} \,,\\
V^i_k  := 
\begin{cases}
V_0                                                                    & \text{ for } i = 0  \\
 \{v \in V_k \, | \, \, supp(v) \subseteq \widehat{\Omega}^i_k\} \quad & \text{ for } i = 1, \dots, p_k  
\end{cases} \,, \\
 R_k^i  := (A_k^i)^{-1} \,. 
\end{cases}
 \end{align}
\end{definition}
We point out that the $V_k$ defined in \eqref{spaces_smoother_app1} satisfy the nestedness condition \eqref{mg_spaces}.
The following lemma describes a decomposition of $V_k$.
\begin{lemma} \label{app1_Vk}
 Given $V_k$ and $V^i_k$ in Definition \ref{def_VR_1}, we have
\begin{equation*}
V_k = \sum \limits_{i=0}^{p_k}  V^i_k\,.
\end{equation*}
Moreover, if we denote with $ v_i \in V^i_k $ the components of any $v \in V_k$ 
(i.e., such that $v = \sum\limits_{i=0}^{p_k} v_i$),
 there is a constant $C_0$ independent of $h_0$, $h_k$ and $p_k$ such that  
\begin{equation} \label{ineq_C_0}
 \sum\limits_{i=0}^{p_k} a(v_i , v_i) \leq C_0 \, a(v,v)  \quad \forall v \in V_k\,.
\end{equation}
\end{lemma}
\begin{proof}
A proof of this result can be found in \cite{xu1992iterative} and \cite{Dryja}. 
\end{proof}
The choice of $R_k^i$ implies that $ R^i_k A^i_k = I $ for all $i=0, \ldots, p_k$ and $k=0, \ldots, J$
and so we have $w_{1,k} = w_1 = 1$.
Assumption \ref{RiSPD_w1lt2} is then satisfied.
Now we can look at Assumptions \ref{exist_K0} and \ref{exist_K1} 
by showing the existence of the parameters $K_0$ and $K_1$ for this application.
\begin{lemma}\label{app1_K0}
Let $ V_k $ be as in Definition \ref{def_VR_1}. 
Then, there exists a constant $K_0$ satisfying Assumption \ref{exist_K0}.
\end{lemma}
\begin{proof}
Let $ v \in V_k $ and consider the decomposition of $V_k$ provided by Lemma \ref{app1_Vk}. Then 
we have by \eqref{ineq_C_0}
\begin{align*}
 \sum_{i=0}^{p_k} ((R^i_k)^{-1}v_i , v_i) &= \sum_{i=0}^{p_k} (A^i_k v_i , v_i) 
 = \sum_{i=0}^{p_k} a(v_i , v_i) 
 \leq C_0 \, a(v,v) \,.
\end{align*}
This shows that $K_0$ exists and $ K_0 = C_0 $.
\end{proof}

\begin{lemma}\label{app1_K1}
Let $ V^i_k $ and $ R^i_k $ as in Definition \ref{def_VR_1}. 
Then, there exists a constant $K_1$ satisfying Assumption \ref{exist_K1}.
\end{lemma}
\begin{proof}
Here we follow a variation of a procedure in \cite{mathew2008domain}, Section 2.5.
Given the subdomains $ \widehat{\Omega}_k^1, \dots, \widehat{\Omega}_k^{p_k}$ as in Definition \ref{def_subd_1},
we define the symmetric $ p_k \times p_k $ matrix $G$ by
\begin{align} \label{g_0}
G_{i \, j} =
\bigg \{
\begin{array}{rl}
 1 &  \quad \mbox{if} \quad \widehat{\Omega}_k^i \cap \widehat{\Omega}_k^j \neq \emptyset,\\
 0 &  \quad \mbox{if} \quad \widehat{\Omega}_k^i \cap \widehat{\Omega}_k^j   = \emptyset\\
\end{array} \qquad g_0 = \max \limits_{i=1,\dots,p_k} \Big(\sum \limits_{j=1}^{p_k} G_{i \, j} \Big) = || G ||_{\infty}  
\end{align}
Note that $g_0$ represents the maximum number of neighbors intersecting a subdomain (counting self intersections) 
and it does not depend on $p_k$ but only on the geometry of the triangulation.
Unlike \cite{mathew2008domain}, 
the summation in the definition of the constant $ g_0 $ does not exclude the $ i $ term.
Also, by the choice of the subdomains, $g_0$ will be uniformly bounded.
Let $S \subset \{0,1, \dots , p_k\} \times \{0,1, \dots , p_k\}$, and consider the decomposition as in \cite{mathew2008domain}, namely
\begin{displaymath}
  \begin{split}
     S  & = S_{00} \cup S_{10} \cup S_{01} \cup S_{11}\,, \\
 S_{00} & = \{(i,j) \in S \, | \, i = 0 \, , \, j = 0  \} \,, \\
 S_{10} & = \{(i,j) \in S \, | \, 1 \leq  i  \leq p_k \, , \, j = 0  \}  \,, \\
 S_{01} & = \{(i,j) \in S \, | \, i = 0 \, , \, 1 \leq   j \leq p_k  \}  \,, \\
 S_{11} & = \{(i,j) \in S \, | \, 1 \leq  i \, , \, j \leq p_k  \}  \,. 
  \end{split}
 \end{displaymath}
Let $u_i , v_i \in V_k$ for $i = 0,1, \dots , k$, then the sum over $S$ can be split as 
\begin{align} \label{sum49}
 \sum_{(i,j) \in S} |a(T^i u_i \,,\, T^j v_j)| &= \sum_{(i,j) \in S_{00}}|a(T^i u_i , T^j v_j)| + \sum_{i:(i,0) \in S_{10}} |a(T^i u_i , T^0 v_0)| \\
 & + \sum_{j:(0,j) \in S_{01}} |a(T^0 u_0 , T^j v_j)| +  \sum_{(i,j) \in S_{11}} |a(T^i u_i \,,\, T^j v_j)| \nonumber
\end{align}
 Let us consider one summand at a time. By the Cauchy-Schwarz inequality and Lemma \ref{lemma_Ti} 
 we have that
\begin{align*}
\Big( \sum_{(i,j) \in S_{00}} |a(T^i u_i , T^j v_j)| \Big)^2 & 
\leq \Big(\sum_{i=0}^{p_k}a(T^i u_i \,,\, u_i)\Big)  \Big(\sum_{j=0}^{p_k}a(T^j v_j \,,\, v_j)\Big) \,. 
 \end{align*}
 If for given $i$ and $j$, $\widehat{\Omega}_k^i \cap \widehat{\Omega}_k^j = \emptyset $, then $ a(T^i u_i , T^j v_j) = 0 $.
 Hence for the last summand we have
\begin{equation} \label{ineq_S11}
\begin{split}
 & \Big ( \sum_{(i,j) \in S_{11}} |a(T^i u_i \,,\, T^j v_j)| \Big)^2  \\
 & =  \Big ( \sum_{(i,j) \in S_{11}} G_{i j} \, |a(T^i u_i \,,\, T^j v_j)| \Big)^2  \\
 & \leq \Big( \sum_{(i,j) \in S_{11}} G_{ij}\sqrt{a(T^i u_i \,,\, T^i u_i)}\sqrt{a(T^j v_j \,,\, T^j v_j)}\Big)^2 \\
 & = \Big( \sum_{(i,j) \in S_{11}} G_{ij}\sqrt{a(T^i u_i \,,\, u_i)}\sqrt{a(T^j v_j \,,\, v_j)}\Big)^2 \text{ (by def of $P^i_k$)} \\%
  & \leq \Big( \sum_{i=1}^{p_k} \sum_{j=1}^{p_k} G_{ij}\sqrt{a(T^i u_i \,,\, u_i)}\sqrt{a(T^j v_j \,,\, v_j)}\Big)^2 \\
 & \leq \rho(G)^2 \, \Big(\sum_{i=1}^{p_k}a(T^i u_i \,,\, u_i)\Big)  \Big(\sum_{j=1}^{p_k}a(T^j v_j \,,\, v_j)\Big) \text{ (by (4.12) in \cite{olshanskii2014iterative}) } \\%
 & \leq g_0^2 \, \Big(\sum_{i=0}^{p_k}a(T^i u_i \,,\, u_i)\Big)  \Big(\sum_{j=0}^{p_k}a(T^j v_j \,,\, v_j)\Big), 
\end{split}
\end{equation}
 where $\rho(G)$ denotes the spectral radius of $G$ which satisfies $\rho(G) \leq ||G||_{\infty}$.\\
 Considering that $a(T^iu_i , T^0v_0)=0$ anytime $\widehat{\Omega}_k^i \cap \widehat{\Omega}_k^0 = \emptyset $, 
 for the second term of the sum in \eqref{sum49} we have
\begin{align*}
\Big( \sum_{i:(i,0) \in S_{10}} |a(T^iu_i , T^0v_0)| \Big)^2  &\leq \Big( \sum_{i:(i,0) \in S_{10}} G_{0i}\sqrt{a(T^iu_i \, , \, T^iu_i )} \sqrt{a(T^0v_0 , T^0v_0)} \Big)^2  \\
& = \Big( \sum_{i:(i,0) \in S_{10}} G_{0i}\sqrt{a(T^iu_i \, , \, T^iu_i )}  \Big)^2 \,\, a(T^0v_0 , T^0v_0) \\
& = \Big( \sum_{i:(i,0) \in S_{10}} G_{0i}\sqrt{a(T^iu_i \, , \, T^iu_i )}  \Big)^2 \,\, a(T^0v_0 , T^0v_0) \\
& \leq \Big(\sum_{i=1}^{p_k} G_{0i}\Big) \,\, \Big( \sum_{i:(i,0) \in S_{10}} a(T^iu_i \, , \, T^iu_i )  \Big)\,\, a(T^0v_0 , T^0v_0)\\
& \leq  g_0 \, \Big(\sum_{i=0}^{p_k}a(T^i u_i \,,\, u_i)\Big)  \, \, a(T^0v_0 , v_0) \; \text{ (by def of $P^i_k$ and $g_0$)}  \\
& \leq  g_0 \, \Big(\sum_{i=0}^{p_k}a(T^i u_i \,,\, u_i)\Big)  \, \, \Big(\sum_{j=0}^{p_k}a(T^jv_j , v_j)\Big) \,.
\end{align*}
Similarly, for the last term of the sum we have
\begin{align*}
\Big( \sum_{j:(0,j) \in S_{01}} |a(T^0u_0 , T^jv_j)| \Big)^2
\leq  g_0 \, \Big(\sum_{i=0}^{p_k}a(T^i u_i \,,\, u_i)\Big)  \, \, \Big(\sum_{j=0}^{p_k}a(T^jv_j , v_j)\Big) \,.
\end{align*}
Combining these four inequalities, it follows that
\begin{align*}
 \Big (\sum_{(i,j) \in S} |a(T^i u_i \,,\, T^j v_j)| \Big)^2 
 \leq 4 \, (1 + 2 g_0 + g_0^2) \, \Big(\sum_{i=0}^{p_k}a(T^i u_i \,,\, u_i)\Big)  \Big(\sum_{j=0}^{p_k}a(T^j v_j \,,\, v_j)\Big) \,.
\end{align*}
This shows that $K_1$ exists and 
\begin{align}
 K_1 = 2 \, (1+g_0).
\end{align}
\end{proof}


The next result follows immediately from Lemmas \ref{app1_K0} and \ref{app1_K1}.  
It shows how the assumption about the non-increasing behavior of $\psi_k$ in \eqref{deltak_choice} is satisfied.
\begin{lemma}
Let $ V^i_k $ and $ R^i_k $ as in Definition \ref{def_VR_1}. 
Then Assumption \ref{smooth_psik_nonincr} is satisfied with
\begin{align}
 \delta_{k} = 1 - \dfrac{1}{C_0(3+ 2 \, g_0)^2}\,, \qquad \psi_{k} =  \dfrac{m_k}{C_0(3+ 2 \, g_0)^2}\,,
\end{align}
if and only if $m_k$ is non-increasing.
Here, $C_0$ is the constant from Lemma \ref{app1_Vk} and $g_0$ is defined in \eqref{g_0}. 
\end{lemma}
Notice that the constant $\delta_{k}$ is independent of $k$.
Hence, we have the convergence result.
\begin{theorem}
 If $m_k$ is non-increasing, the multigrid algorithm \ref{alg_Vcycle} converges with 
\begin{align}
 \gamma_k = \dfrac{C_0 (3 + 2 \, g_0)^2}{C_0 (3 + 2 \, g_0)^2 + 2 \, m_k } \,,
\end{align}
where $ \gamma_k $ are the constants defined in \eqref{gamma_k}.

Moreover, if $m_1 = m_2 = \ldots = m_J$, the error bound is optimal 
in the sense that it does not deteriorate as the number of multigrid spaces $J$ increases.
\end{theorem}
Notice that convergence can be achieved even by performing only one smoothing iteration, 
but a larger $ m_k $ can further lower the error bound. 

\subsection{Local refinement: 
overlapping nested subdomains,
subproblems on regular grids 
and approximate subsolvers} \label{sec_loc_reg}

Now we move to an application involving a locally refined grid.
\giachi{\begin{definition}[Triangulations $\mathcal{T}_k$] \label{def_triang_3}
Let $\{ \Omega_k\}_{k=0}^{J}$ be a collection of closed subdomains of $\Omega$ such that
$$ \Omega_J \subset \Omega_{J-1} \subset \cdots \Omega_0 \equiv \Omega.$$
Let $\mathcal{T}_0$ be a coarse quasi-uniform triangulation of $\Omega$ of size $h_0\in (0,1]$. 
Assume $\mathcal{T}_{k-1}$ has been defined, then $\mathcal{T}_k$ is obtained performing midpoint refinement only on those elements
of $\mathcal{T}_{k-1}$ that belong to $\Omega_k$.
\end{definition}
This process introduces hanging nodes, causing the grid $\mathcal{T}_k$ to become
irregular, for all $k=1,\ldots,J$. However, restricted to $\Omega_k$, $\mathcal{T}_k$ is a regular grid
without hanging nodes and size $h_k =  2^{-k} h_0$.
We observe that the sequence $\{\mathcal{T}_k\}_{k=0}^J$ is nested in the sense that an element $T \in \mathcal{T}_{k-1}$
can be written as the union of elements in $\mathcal{T}_{k}$ \cite{Ciarlet:2002:FEM:581834}.
}
Moreover, by construction we have that $h_0 = \max \limits_{T \in \mathcal{T}_k} h_{T}$, where $h_{T}$ denotes the size of one element $T \in \mathcal{T}_k$.
 Figure \ref{local_case_1} sketches an example of triangulation for this case.
\begin{figure}[h]
\centering
\includegraphics[scale=0.5]{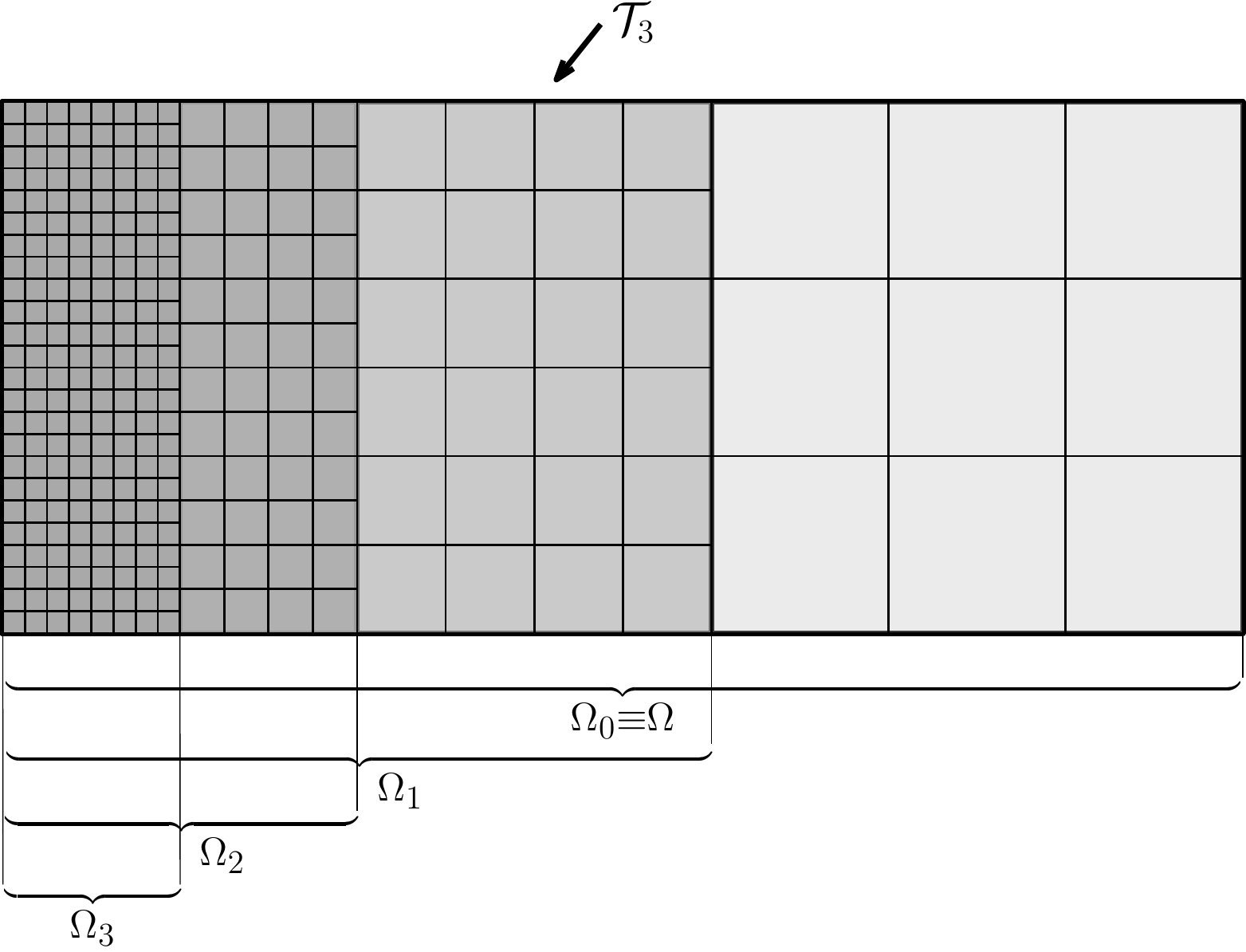}
\caption{Subdomains involved in the construction of the irregular triangulation obtained with local midpoint refinement.}\label{local_case_1}
\end{figure}
Concerning the spaces $ V_k $, the subspaces $ V^i_k $ and the corresponding subsolvers $ R^i_k $
we choose the following.
\begin{definition} \label{def_VR_2}
Given the overlapping subdomains $\Omega_i$ 
and the triangulations $ \mathcal{T}_k $ in Definition \ref{def_triang_3}, 
we set for $ k = 0, \ldots, J $ and for $i=0, \ldots, k $
\begin{align}\label{app2_summary}
\begin{cases}
 V_k  = \{ v \in H_0^1(\Omega) \cap C^0(\Omega) : v|_{\tau} \in \mathcal{P}_1, \, \, 
 \forall \tau \in \mathcal{T}_k \,,  \}, \\
 \quad \quad \, \text{where $\mathcal{T}_k$ is as in Definition \ref{def_triang_3}} \,, 
 \\
 V^i_k  :=
 \begin{cases}
 V_0 \,,                                                    & i = 0 \,,   \\
 \{ v \in V_{i} \,\, | \, \, supp(v) \subseteq \Omega_i\} \,, & i = 1, \ldots, k 
 \end{cases} \,, \\
 \vspace{0.001em} \\
 R^i_k  := 
 \begin{cases}
     A_0^{-1}  \,,                       & i = 0 \,,  \\ 
     \dfrac{1}{\lambda^i \, k} \, I  \,, & i = 1, \ldots, k \,  
 \end{cases}
\end{cases}
\end{align}
where $\lambda^i$ denotes the spectral radius of $A^i$.
\end{definition}
We point out that the $V_k$ satisfy by construction the nestedness condition \eqref{mg_spaces}.
\giachi{Moreover, the continuity requirement in the definition of $V_k$ implies that its nodal
basis will have no function associated to hanging nodes of $\mathcal{T}_k$}.
Also, since the support of the functions in each $V^i_k$ is contained in $ \Omega_i $, 
 the subproblems are all defined on uniformly refined grids without hanging nodes,
 although $ \mathcal{T}_k$ is irregular.
\giachi{This considerably simplifies the implementation since no actual constraints have to be added 
and no change in the nodal basis is required to obtain a continuous numerical solution}.
If $ v \in V^i_k $, it will be a linear combination of the basis functions 
associated with the interior nodes of $\Omega_i$.
The following lemma is a consequence of the choice of the subspaces introduced in Definition \ref{def_VR_2}. 
See also \cite{dryja1989optimality} for more on this decomposition.
\begin{lemma}
 Given $ V_k $ and $ V^i_k $ in Definition \ref{def_VR_2}, we have
$$ V_k = \sum_{i=0}^k V^i_k \,. $$
\end{lemma}
\begin{proof}
The result follows if for given $ v \in V_k $ we can find a decomposition 
$ v = \sum \limits_{i=0}^k v_i $ such that $ v_i \in V^i_k $.
 To do this, we will consider a result from \cite{bramble1991convergence} 
that relies on the construction of a sequence of operators $ \widehat{Q}_i : V_k \rightarrow V_i $.
Let $\overline{V}_i$ be the space obtained by taking $\Omega_0 = \Omega_1 = \dots = \Omega_i$, 
namely the space built over a uniformly refined triangulation
of size $ h_i = h_0 2^{-i} $ and let $\overline{Q}_i$ be the $L^2(\Omega)$ projection operator onto $\overline{V}_i$.
Set $\widehat{Q}_k = I$ and for $i=0, \dots, k-1$ define $ \widehat{Q}_i v = w $
as the unique function on $V_i$ that satisfies
\begin{displaymath}
w =
\bigg \{
\begin{array}{rl}
 \overline{Q}_i v &  \qquad \mbox{at the nodes of $V_i$ in the interior of $\Omega_{i+1}$},\\
 v &  \quad \quad \, \mbox{at the remaining nodes of $V_i$} \,. 
\end{array}
\end{displaymath}
It has been shown in \cite{bramble1991convergence} that $ (\widehat{Q}_i - \widehat{Q}_{i-1}) v $ 
is a function in $ V^i_k $ for all $ i = 1, \ldots, k $ and that 
\begin{equation}  \label{brambleineq1}
 \begin{split}
 ((\widehat{Q}_i - \widehat{Q}_{i-1})v , (\widehat{Q}_i - \widehat{Q}_{i-1})v) \leq \widetilde{C}_1 \, \, h^2_0
 \,\,a(v,v) \quad \mbox{for} \, \, i =1, \dots, k  \,, \\
 ((\widehat{Q}_i - \widehat{Q}_{i-1})v , (\widehat{Q}_i - \widehat{Q}_{i-1})v) \leq C_1 \, \, {\lambda_i}^{-1} \,\,a(v,v) \quad \mbox{for} \, \, i =1, \dots, k \,, \\
  a(\widehat{Q}_i v , \widehat{Q}_i v) \leq C_2 \, a(v,v) \quad \mbox{for} \, \, i = 0, \dots, k-1 \,,
 \end{split}
\end{equation}
where $\lambda_i$ denotes the spectral radius of the operator $A_i$ and $\widetilde{C}_1$, $C_1$ and $C_2$ do not depend on $i$.
It then follows that for all $ v \, \in \, V_k$
\begin{align} \label{bramble_dec}
 v = \widehat{Q}_0 v + \sum_{i=1}^k (\widehat{Q}_i - \widehat{Q}_{i-1}) v = \sum_{i=0}^k v_i  \,,
\end{align}
where
\begin{equation} \label{vi_dec}
  v_i :=
 \begin{cases}
  \widehat{Q}_0 v \,,                        & i = 0 \,, \\
 (\widehat{Q}_i - \widehat{Q}_{i-1}) v \,,   & i = 1, \dots, k \,,
  \end{cases}
 \end{equation}
 with $ v_i\in V^i_k $ for all $i$.
\end{proof}

\begin{remark}
In this case $p_k = k$.
This means that at each level $k$, the number of subdomains is fixed and equal to $ k $ as well.
At the given level $k$, notice that $ V^i_k \nsubseteq V^j_k, \forall i>j $, 
 since the trace of $  V^i_k $ on $ \partial \Omega_i $ is zero
 while the trace of $V^j_k$ is not.
Moreover, it follows \giachi{from Definition \ref{def_VR_2}} that the $  V^i_k  $ are independent of $ k $.
Consequently so will be $A^i$, in the sense that $A^i_i = A^i_{i+1} = \ldots \ = A^i_k$.
\end{remark}
Note that with the choice of $ R^i_k $ in \eqref{app2_summary} we have $ \rho(R_k^0 A_0) = 1 $ and $ \rho(R^i_k A^i) = 1/k $
for all $ i = 1, \ldots, k $. This implies that $w_{1,k} = w_1 = 1$, 
so that Assumption \ref{RiSPD_w1lt2} is satisfied.
Now we can show the existence of the parameters $K_0$ and $K_1$.
\begin{lemma}
Let $ V^i_k $ and $ R^i_k $ as in Definition \ref{def_VR_2}. 
Then, there exists a constant $K_0$ satisfying Assumption \ref{exist_K0}.
\end{lemma}
\begin{proof}
 Using the definition of $ R_k^i $ together with \eqref{brambleineq1} and \eqref{vi_dec} we have
\begin{align*}
 \sum_{i=0}^k ((R_k^i)^{-1} v_i, v_i) 
 & = (A_0 v_0 , v_0) +  k \, \sum_{i=1}^k \lambda^i \, ((\widehat{Q}_i - \widehat{Q}_{i-1})v , (\widehat{Q}_i - \widehat{Q}_{i-1})v) \\
 & \leq a(v_0 , v_0) + k \, \sum_{i=1}^k C_1 \dfrac{\lambda^i}{\lambda_i} \,\,a(v,v) \\
 & \leq C_2 \, a(v,v) + k^2 \, C_1 \, a(v,v) \\
 & \leq \max\{C_1 , C_2\}(1 +\, k^2) \, a(v,v) = C_3 (1 + \, k^2) \, a(v,v) \,. 
 \end{align*}
This shows that $K_0$ exists and 
\begin{align}\label{C_3}
 K_0 = C_3 \, (1 + k^2).
\end{align}
\end{proof}


Let us now show the existence of $K_1$ for this application.
\begin{lemma}
Let $ V^i_k $ and $ R^i_k $ as in Definition \ref{def_VR_2}.
Then, there exists a constant $K_1$ satisfying Assumption \ref{exist_K1}.
\end{lemma}
\begin{proof}
For $ i = 1, \dots, k $ and $ u \in V_k $ we have
\begin{align*}
0 \leq \,  a(T^i u \,,\, T^i u) &= a(R_k^iQ^iA_k \, u \,,\, R_k^iQ^iA_k \, u)  \\
&= \Big( \dfrac{1}{\lambda^i \, k} \Big)^2 \, a(Q^iA_k \, u \,,\, Q^iA_k \, u)  \\
&= \Big( \dfrac{1}{\lambda^i \, k}\Big)^2 \, (Q^iA_k \, u \,,\, A^iQ^iA_k \, u)  \qquad \mbox{(def. of $A^i$)}\\
&\leq \Big( \dfrac{1}{\lambda^i \, k}\Big)^2 \lambda^i \, (Q^iA_k \, u \,,\, Q^iA_k \, u)  \qquad \mbox{($A^i$ is SPD wrt $(\cdot,\cdot)$)}\\
&= \Big( \dfrac{1}{\lambda^i \, k^2}\Big) \, (Q^iA_k \, u \,,\, A_k \, u)  \qquad \qquad \mbox{(def. of $Q^i$)}\\
&=  \dfrac{1}{k} \, (R_k^iQ^iA_k \, u \,,\, A_k \, u)  \qquad \qquad \quad \, \, \,\, \mbox{(def. of $R_k^i$)}\\
&=  \dfrac{1}{k} \, (T^i \, u \,,\, A_k \, u)  \qquad \qquad \qquad \quad \quad \, \mbox{(def. of $T^i$)}\\
&=  \dfrac{1}{k} \, a(T^i \, u \,, \, u)  \qquad \qquad \qquad \qquad \quad \mbox{(def. of $A_k$)} \,.
\end{align*}
For $i=0$ we have $ R_k^0 = A_0^{-1} $ so that $ T^0 = P^0 $ and 
\begin{align}
 a(T^0 u , T^0 u) &= a(P^0 u , P^0 u) = a(P^0 u ,  u) = a(T^0 u ,  u) \,.
\end{align}
 In summary
\begin{equation} \label{ineqT0Ti}
 a(T^i u\,,\, T^i u) 
 \begin{cases}
  = \, a(T^i u ,  u)                        \,, & i = 0 \,, \\
\leq  \dfrac{1}{k} \, a(T^i \, u \,, \, u)  \,, & i = 1, \dots, k \,.
 \end{cases}
\end{equation}
 Let $S \subset \{0,1, \dots , k\} \times \{0,1, \dots , k\}$,
 and consider the decomposition of such set as before with $ p_k = k $, namely
\begin{equation}
\begin{split}
     S  & = S_{00} \cup S_{10} \cup S_{01} \cup S_{11}\,, \\
 S_{00} & = \{(i,j) \in S \, | \, i = 0 \, , \, j = 0  \} \,, \\
 S_{10} & = \{(i,j) \in S \, | \, 1 \leq  i  \leq k \, , \, j = 0  \} \,,\\
 S_{01} & = \{(i,j) \in S \, | \, i = 0 \, , \, 1 \leq   j \leq k  \} \,, \\
 S_{11} & = \{(i,j) \in S \, | \, 1 \leq  i \, , \, j \leq k  \} \,. 
\end{split}
\end{equation}
Let $u_i , v_i \in V_k$ for $i = 0,1, \dots , k$, then
\begin{align}\label{sum49}
 \sum_{(i,j) \in S} |a(T^i u_i \,,\, T^j v_j)| &= \sum_{(i,j) \in S_{00}} |a(T^i u_i , T^j v_j)| + \sum_{i:(i,0) \in S_{10}} |a(T^i u_i , T^0 v_0)| \\
 &+  \sum_{j:(0,j) \in S_{01}} |a(T^0 u_0 , T^j v_j)| +  \sum_{(i,j) \in S_{11}} |a(T^i u_i \,,\, T^j v_j)| \,. \nonumber
\end{align}
Let us consider one summand at a time.
By the Cauchy-Schwarz inequality with the $ a(\cdot,\cdot) $ inner product,
\eqref{ineqT0Ti} and Lemma \ref{lemma_Ti} we have that
\begin{align*}
\Big( \sum_{(i,j) \in S_{00}} |a(T^i u_i , T^j v_j)| \Big)^2 &\leq a(T^0 u_0 ,  u_0) \, a(T^0 v_0 ,  v_0) \\
&\leq \Big(\sum_{i=0}^{k}a(T^i u_i \,,\, u_i)\Big)  \Big(\sum_{j=0}^{k}a(T^j v_j \,,\, v_j)\Big) \,.  
\end{align*}
 For the last summand we have,
 using again the same properties,  
\begin{align*}
 \Big ( \sum_{(i,j) \in S_{11}} |a(T^i u_i \,,\, T^j v_j)| \Big)^2 & \leq  \Big( \sum_{(i,j) \in S_{11}} \sqrt{a(T^i u_i \,,\, T^i u_i)} \, \sqrt{a(T^j v_j \,,\, T^j v_j)}\Big)^2  \\
 \leq \dfrac{1}{k^2} \, & \Big( \sum_{i:(i,j) \in S_{11}}\sqrt{a(T^i u_i \,,\, u_i)} \Big)^2  \Big( \sum_{j:(i,j) \in S_{11}}\sqrt{a(T^j v_j \,,\, v_j)} \Big)^2 
\\
& \leq \dfrac{1}{k^2} \,\Big( \Big(\sum_{i=1}^{k}a(T^i u_i \,,\, u_i)\Big) \, k \Big) \Big( \Big(\sum_{j=1}^{k} a(T^j v_j \,,\, v_j)\Big)\, k \Big) \\
  &\leq \Big(\sum_{i=0}^{k}a(T^i u_i \,,\, u_i)\Big)  \Big(\sum_{j=0}^{k}a(T^j v_j \,,\, v_j)\Big) \,.
 \end{align*}
For the second term of the sum in \eqref{sum49} use 
the Cauchy-Schwarz inequality with the $ a(\cdot,\cdot) $ inner product, 
and \eqref{ineqT0Ti},
 \begin{align*}
\Big( \sum_{i:(i,0) \in S_{10}} |a(T^iu_i , T^0v_0)| \Big)^2 
& \leq \Big( \sum_{i:(i,0) \in S_{10}} \sqrt{a(T^iu_i \, , \, T^iu_i )} \sqrt{a(T^0v_0 , T^0v_0)} \Big)^2  \\
& = \Big( \sum_{i:(i,0) \in S_{10}} \sqrt{a(T^iu_i \, , \, T^iu_i )}  \Big)^2 \,\, a(T^0v_0 , T^0v_0) \\
& \leq \Big(\sum_{i:(i,0) \in S_{10}} 1 \Big) \Big(\sum_{i:(i,0) \in S_{10}} a(T^iu_i \, , \, T^iu_i ) \Big) \,\, a(T^0v_0 , v_0) \\
& \leq k \, \dfrac{1}{k}\Big(\sum_{i:(i,0) \in S_{10}} a(T^iu_i \, , \, u_i ) \Big) \,\, a(T^0v_0 , v_0) \\
& \leq \Big(\sum_{i=0}^{k}a(T^i u_i \,,\, u_i)\Big)  \Big(\sum_{j=0}^{k}a(T^j v_j \,,\, v_j)\Big) \,. 
\end{align*}
Similarly, for the last term of the sum we have
\begin{align}
\Big( \sum_{j:(0,j) \in S_{01}} |a(T^0u_0 , T^jv_j)| \Big)^2 &\leq \Big(\sum_{i=0}^{k}a(T^i u_i \,,\, u_i)\Big)  \Big(\sum_{j=0}^{k}a(T^j v_j \,,\, v_j)\Big) \,. \nonumber
\end{align}
Combining these four inequalities, it follows that
\begin{align}
 \Big (\sum_{(i,j) \in S} |a(T^i u_i \,,\, T^j v_j)| \Big)^2 \leq 4 \, \Big(\sum_{i=0}^{k}a(T^i u_i \,,\, u_i)\Big)  \Big(\sum_{j=0}^{k}a(T^j v_j \,,\, v_j)\Big) \,.
\end{align}
This shows that $K_1$ exists and 
\begin{align}
 K_1 = 2.
\end{align}
\end{proof}

\begin{lemma}
Let $ V^i_k $ and $ R^i_k $ as in Definition \ref{def_VR_2}. 
Then, Assumption \ref{smooth_psik_nonincr} is satisfied with
\begin{align}
 \delta_{k} 
 = 1 - \dfrac{1}{C_4 (1 + k^2)}, 
 \qquad \psi_{k} =  \dfrac{m_k}{C_4 (1 + k^2)}\,,
\end{align}
if and only if $ m_k $ is chosen so that $ \psi_k $ is non-increasing.
Here, $C_4 = 9 \, C_3$ and $C_3$ is the constant in \eqref{C_3}.
\end{lemma}
 Note that $\delta_{k}$ is now increasing.
 Various choices of $m_k$ guarantee Assumption \ref{smooth_psik_nonincr}:
 constant $ m_k = 1 $,
 decreasing $ m_k = J + 1 - k $,
 increasing $ m_k = 1 + k $.
 We now state the convergence result.
\begin{theorem}
  If $ m_k $ is chosen so that $ \psi_k $ is non-increasing,
  the multigrid algorithm \ref{alg_Vcycle} converges with 
\begin{align}
 \gamma_k = \dfrac{C_4 (1 + k^2)}{C_4 (1 + k^2) + 2 \, m_k }\,,
\end{align}
where $ \gamma_k $ is defined in Theorem \ref{mg_conv}, $ k = 0, 1, \ldots, J$.

Moreover, the error bound is optimal (in the sense that it does not depend on the number of multigrid spaces $J$)
if and only if $m_k = q (1+k^2)$ for some $ q \in \mathbb{N} $,
and is given by
\begin{align}
 \gamma_1 = \gamma_2 = \dots = \gamma_J = \dfrac{C_4}{2 q + C_4} \,.
\end{align}
\end{theorem}
We observe that the number of smoothing iterations appears in the error bound 
and this was not shown in \cite{bramble1992analysis}.
 \bornia{Although the choice $ m_k = q (1+k^2) $ is not optimal in terms of computational cost,
   since more smoothing steps are needed on finer grids,
   nevertheless it guarantees that the error bound is independent of the number of levels.}
 
\subsection{Local refinement: overlapping non-nested subdomains,
subproblems on irregular grids 
and exact subsolvers}


We now describe another local refinement application.
We keep the same triangulations as in Definition \ref{def_triang_3}
and the same definition of $V_k$ \giachi{as in the previous local refinement application}.
However, \giachi{the subdomains are chosen as in Definition \ref{def_subd_1}}.
This will lead to a different characterization of the space $ V_k $.

A sketch of the subdomains involved in this application is visible in Figure \ref{local_case_2}.
Moreover, unlike Section  \ref{sec_loc_reg},
the overlapping subdomains \giachi{$ \widehat{\Omega}^k_i$} at each level $k$ are not nested.
\begin{figure}[h]
\centering
\includegraphics[scale=0.6]{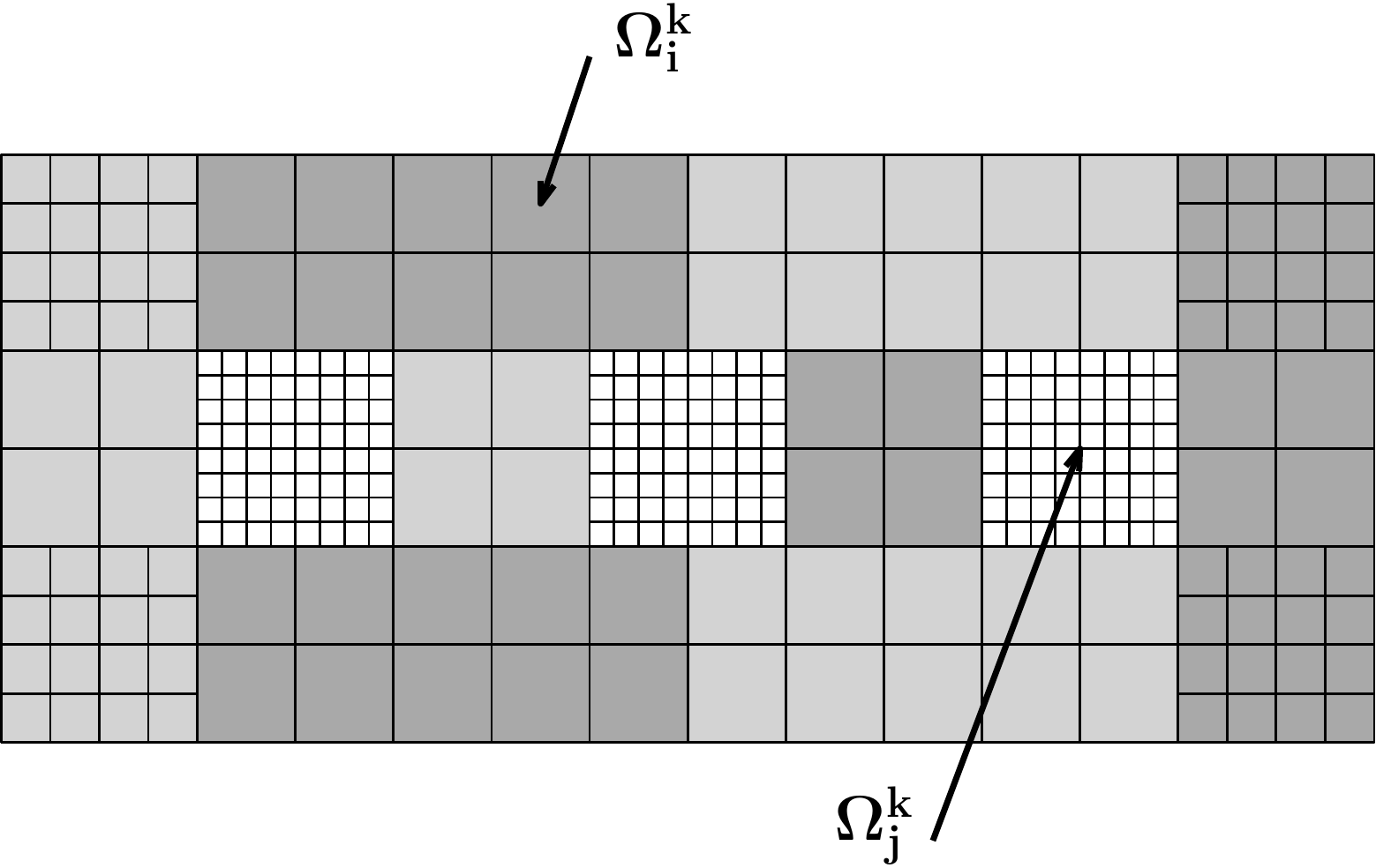}
\caption{\bornia{A subdivision into} non-overlapping subdomains involved in the local refinement application 
\bornia{(different subdomains are identified by a change in the shade of grey)}.}\label{local_case_2}
\end{figure}
\begin{figure}[h]
\centering
\includegraphics[scale=0.6]{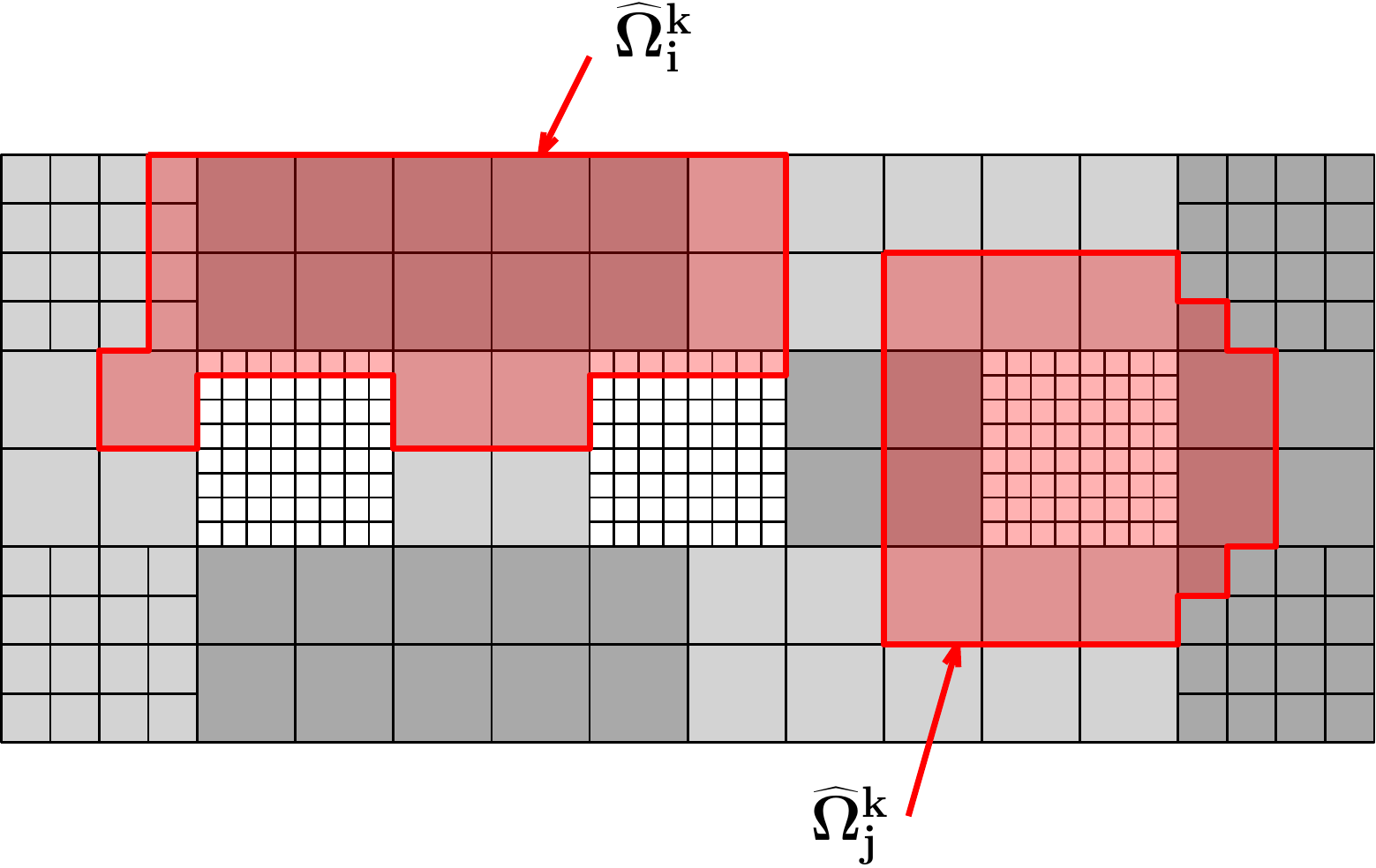}
\caption{\bornia{An example of construction of} overlapping subdomains obtained from the non-overlapping subdomains in Figure \ref{local_case_2}.}\label{local_case_2bis}
\end{figure}
 Let us now define the spaces $ V_k $ and choose the subspaces for its decomposition, and the subsolvers $ R^i_k $.
\begin{definition} \label{def_VR_3}
Given the triangulations $\mathcal{T}_k$ in Definition \ref{def_triang_3} and the subdomains $ \Omega_i $ in Definition \ref{def_subd_1}, 
we set for $ k = 0, \ldots, J $ and for $i=0, \ldots,  p_k $
\begin{align} \label{spaces_smoothers_app3}
\begin{cases}
 V_k  = \{ v \in H_0^1(\Omega) \cap C^0(\Omega) : v|_{\tau} \in \mathcal{P}_1, \, \, \forall \tau \in \mathcal{T}_k \,,  \} \\
 \quad \quad \text{where $\mathcal{T}_k$ is as in Definition \ref{def_triang_3}} \,, 
 \\
 V^i_k  :=
 \begin{cases}
   V_0 \,,                                                 & i = 0 \,, \\
  \{v \in V_k \, | \, \, supp(v) \subseteq \widehat{\Omega}^k_i\}\,,   & i = 1, \dots, p_k \,, \\
\end{cases} \,, \\
 R_k^i  := (A_k^i)^{-1} \,.
\end{cases}
\end{align}
\end{definition}
Since the definition of $ V_k $ in Definition \ref{def_VR_3} coincides with Definition \ref{def_VR_2},
 the $V_k$ again satisfy the nestedness condition \eqref{mg_spaces} \giachi{and the nodal basis does not have any function
 associated to the hanging nodes, because of the continuity requirement}.
Here, we give another characterization of  $ V_k $ based on the non-nested subdomains.
\begin{lemma} \label{app3_Vk}
 Given $ V_k $ and $V^i_k$ in Definition \ref{def_VR_3}, we have
\begin{equation*}
V_k = \sum \limits_{i=0}^{p_k}  V^i_k \,.
\end{equation*}
Moreover, if we denote with $ v_i \in V^i_k $ the components of any $v \in V_k$ (such that $v = \sum\limits_{i=0}^{p_k} v_i$),
then there is a constant $C_k$ dependent only on $k$ such that  
$$ \sum\limits_{i=0}^{p_k} a(v_i , v_i) \leq C_k \, a(v,v)  \quad \forall v \in V_k\,.
$$
\end{lemma}
\begin{proof}
We are going to construct a set of functions $ \{v_i\}_{i=0}^{p_k} \in V^i_k \subseteq  V_k $
such that every $ v \in V_k $ can be expressed as their sum.
To this end, let $\{\theta^k_i\}_{i=1}^{p_k}$ be a smooth partition of unity subordinate to the cover $\{\widehat{\Omega}^k_i\}_{i=1}^{p_k}$.
This means that $ \sum_{i=1}^{p_k} \theta^k_i = 1 $, $0 \leq \theta^k_i(x) \leq 1$ for all $x \in  \widehat{\Omega}^k_i$ and $supp(\theta^k_i) \subset \widehat{\Omega}^k_i$,
for all $i=1 \dots, p_k$. 
Let $\widehat{V}^j$ be the subspace of $V_k$ defined in Definition \ref{def_VR_2} in the previous local refinement application.
Then we know that $V_k = \sum_{j=0}^{k} \widehat{V}^j$, so that any $v$ in $V_k$ can be written as $v = \sum_{j=0}^{k} \widehat{v}_j$, 
where $\widehat{v}_j \in \widehat{V}^j$ are given by \eqref{vi_dec}.
\giachi{Define $\mathcal{I}^j_h$ to be the standard nodal interpolant of the finite element space 
$\widehat{V}^j$ for all $j=1,\ldots,k$.
Note that this is well defined since each $\widehat{V}^j$ is built on a quasi-uniform grid.}
Then, for $v \in V_k$, set
\begin{equation}
 v_0 =  \widehat{v}_0, \, \quad v_i = \sum_{j=1}^k \mathcal{I}^j_h( \theta^k_i \, \widehat{v}_j)\,, \quad i = 1, \ldots,  p_k \,.    
\end{equation}
Notice that all the terms in the sum that defines $v_i$ are functions in $\widehat{V}^j \subset V_k$ and  have support in $\widehat{\Omega}^k_i$, therefore they all belong to $V_k^i$. 
Moreover, using the fact that the $\mathcal{I}^j_h$ are linear and projections we have
\begin{align*}
v &= \sum_{j=0}^k \widehat{v}_j = \widehat{v}_0 + \sum_{j=1}^k \widehat{v}_j 
=  \widehat{v}_0 + \sum_{j=1}^k \mathcal{I}^j_h ( \widehat{v}_j )  
=  \widehat{v}_0 + \sum_{j=1}^k \mathcal{I}^j_h (\sum_{i=1}^{p_k}\theta^k_i \widehat{v}_j ) 
\\
&= \widehat{v}_0  + \sum_{j=1}^k\sum_{i=1}^{p_k} \mathcal{I}^j_h( \theta^k_i \, \widehat{v}_j) = \sum_{i=0}^{p_k} v_i.
\end{align*}
To prove the second part of the lemma, let us proceed one summand at a time. 
If $T \in \mathcal{T}_k \cap \widehat{\Omega}^k_{i}$, then using an inverse estimate (see \cite{Brenner}) we get
\begin{align*}
| \mathcal{I}^j_h( \theta^k_i \, \widehat{v}_j) |^2_{H^1(T)} &\leq h_0^{-2} \, || \mathcal{I}^j_h( \theta^k_i \, \widehat{v}_j) ||^2_{L^2(T)} \\
& \leq h_0^{-2} C ||  \theta^k_i \, \widehat{v}_j ||^2_{L^2(T)} \\
& \leq h_0^{-2} C || \widehat{v}_j ||^2_{L^2(T)},
\end{align*}
where the constant $C$ is the bound for the operator norm of $\mathcal{I}^j_h$ and it only depends on the reference element \cite{Brenner}.
Summing over all $T \in \mathcal{T}_k \cap \widehat{\Omega}^k_i$ (remember that we assumed the subdomains align with the triangulation) we obtain
\begin{align*}
 |v_i|^2_{H^1(\Omega)} &= |v_i|^2_{H^1(\widehat{\Omega}^k_i)} \leq \sum\limits_{T \in \mathcal{T}_k \cap \widehat{\Omega}^k_i} \Big(\sum_{j=1}^k |\mathcal{I}^j_h( \theta^k_i \, \widehat{v}_j)|_{H^1(T)} \Big)^2   \\
 & \leq \sum\limits_{T \in \mathcal{T}_k \cap \widehat{\Omega}^k_i} k \,\, \sum_{j=1}^k |\mathcal{I}^j_h( \theta^k_i \, \widehat{v}_j)|^2_{H^1(T)} \\ 
 & \leq  k \,\, \sum_{j=1}^k  h_0^{-2} C || \widehat{v}_j ||^2_{L^2(\widehat{\Omega}^k_i)} \,. 
 \end{align*}
 Summing over the subdomains $\widehat{\Omega}^k_i$, and considering that each point in $\Omega$ is covered only a finite number of times \cite{Dryja} we obtain
 \begin{align}
 \sum\limits_{i=1}^{p_k} |v_i|^2_{H^1(\Omega)} \leq  k \,\, \sum_{j=1}^k  h_0^{-2} \widehat{C} || \widehat{v}_j ||^2_{L^2(\Omega)} \,. 
 \end{align}
Thanks to \eqref{brambleineq1} we can say that
\begin{align}
|| \widehat{v}_j ||^2_{L^2(\Omega)} \leq \widetilde{C}_1 \,\, h_0^2 \,\,|v|^2_{H^1(\Omega)}.
\end{align}
Therefore, using the previous results and the Poincar\'e inequality we have
\begin{align*}
 \sum\limits_{i=1}^{p_k} a(v_i,v_i) \leq \overline{\overline{{C}}} \sum\limits_{i=1}^{p_k} |v_i|^2_{H^1(\Omega)} \leq k^2 \overline{\overline{\overline{{C}}}} |v|^2_{H^1(\Omega)} 
 \leq k^2 \widetilde{C} a(v,v) \,.
\end{align*}
Again by \eqref{brambleineq1} we know that $ a(\widehat{v}_0 , \widehat{v}_0) \leq C_2 a(v , v) $, hence if we let
$\widehat{C_0} = \max \{C_2 , \widetilde{C}\}$ we can conclude with
\begin{align*}
 \sum\limits_{i=0}^{p_k} a(v_i,v_i) \leq  \widehat{C_0} \, (1+k^2) a(v,v) \,.
\end{align*}
\end{proof}
 The proof of this lemma for a uniform refinement case 
 relies on the uniform boundedness of the standard nodal interpolator on $V_k$.
 In the case where an irregular grid is employed a nodal interpolator in the classical sense cannot
 be defined on $V_k$.
 An alternative to the solution we adopted in our proof could be to use interpolation operators
 specifically designed for irregular grids as in \cite{heuveline2004interpolation}.

For the subsolvers we clearly have that $R^i_k A^i = I$ for all $i=0, \dots, p_k$ and so again we have $w_{1,k} = w_1 = 1$.
Assumption \ref{RiSPD_w1lt2} is then true.
However, from a practical point of view, defining problems on irregular grids actually requires the 
implementation of the constraints that make the nodal basis of $V_k$ continuous, as in \cite{fries2011hanging}.
Now we can show the existence of $K_0$ and $K_1$.
\begin{lemma}
Let $ V^i_k $ and $ R^i_k $ as in Definition \ref{def_VR_3}. 
Then, there exists a constant $K_0$ satisfying Assumption \ref{exist_K0}.
\end{lemma}
\begin{proof}
Considering the decomposition of $v$ given by Lemma \ref{app3_Vk}, we have 
\begin{align*}
 \sum_{i=0}^{p_k} ((R^i_k)^{-1}v_i , v_i) & = \sum_{i=0}^{p_k} (A^i_k v_i , v_i) 
 = \sum_{i=0}^{p_k} a(v_i , v_i)  
 \leq \widehat{C_0} (1+k^2) \, a(v,v) \,. 
\end{align*}
This shows that $K_0$ exists and 
\begin{align}\label{C_hat}
 K_0 = \widehat{C_0} (1+k^2).
\end{align}
\end{proof}

\begin{lemma}
Let $ V^i_k $ and $ R^i_k $ as in Definition \ref{def_VR_3}. 
Then, there exists a constant $ K_1 $ satisfying Assumption \ref{exist_K1}.
\end{lemma}
\begin{proof}
The existence of $K_1$ can be carried out exactly as for Lemma \ref{app1_K1} 
concerning the case of uniform refinement.
Therefore $K_1$ exists and 
\begin{align}
 K_1 = 2 \, (1+g_0).
\end{align}
\end{proof}

The next lemma immediately follows.
\begin{lemma}
Let $ V^i_k $ and $ R^i_k $ as in Definition \ref{def_VR_3}. 
Then, Assumption \ref{smooth_psik_nonincr} is satisfied with
\begin{align}
 \delta_{k} = 
  1 - \dfrac{1}{\widehat{C_0}(1+k^2)(3+ 2 \, g_0)^2},
 \qquad \psi_{k} =  \dfrac{m_k}{\widehat{C_0}(1+k^2)(3+ 2 \, g_0)^2}\,,
\end{align}
if and only if $ m_k $ is chosen so that $ \psi_k $ is non-increasing.
Here, $\widehat{C_0}$ is the constant in \eqref{C_hat} and $g_0$ is defined in \eqref{g_0}.
\end{lemma}
The constant $\delta_{k}$ is again increasing.
Consequently the convergence bound for the multigrid algorithm is obtained.
\begin{theorem}
If $ m_k $ is chosen so that $ \psi_k $ is non-increasing,
 the multigrid algorithm \ref{alg_Vcycle} converges with 
\begin{align}
 \gamma_k = \dfrac{C_5 (1 + k^2)}{C_5 (1 + k^2) + 2 \, m_k },
\end{align}
where $ \gamma_k $ is defined in Theorem \ref{mg_conv}
and $ C_5 = \widehat{C_0} \, (3+ 2 \, g_0)^2 $.

Moreover, the error bound is optimal (in the sense that it does not depend on the number of multigrid spaces $J$)
if and only if $m_k = q (1+k^2)$ for some $ q \in \mathbb{N} $,
and is given by
\begin{align}
 \gamma_1 = \gamma_2 = \dots = \gamma_J = \dfrac{C_5}{2 q + C_5}\,.
\end{align}
\end{theorem}

\section{Conclusions}

In this paper we performed a convergence analysis 
of a multigrid algorithm 
for symmetric elliptic PDEs under no regularity assumptions
with smoothers of SSC type.
In particular, we focused on the dependence of the multigrid error bound 
on the number \giachi{of smoothing steps.
This represents a novel result for the case of no-regularity assumptions.
We provided an analysis that can be used for any smoothing procedure of symmetric SSC type.
We then utilized this framework to address  
uniform and local refinement applications
and study convergence bounds for the multigrid error.
Our theory allows an arbitrary number of hanging nodes on a given edge of the triangulation.}
A judicious choice of the subdomain solvers and of the number of smoothing steps at each level
can avoid the dependence of the multigrid error bound on the total number of multigrid levels.
To this end, proper decompositions of the finite element spaces had to be derived in the analysis.
For the uniform refinement case, a uniform bound for the multigrid error can be obtained,
even regardless of the choice of the smoothing steps.
For the local refinement applications, we described 
two different subspace decompositions  of the multigrid space using overlapping nested or non-nested subdomains
that correspond to different ways of enforcing the continuity 
of the finite element space \giachi{when hanging nodes are present}.
In both cases, we show that convergence can be obtained and optimality can be guaranteed 
by appropriately choosing the number of smoothing steps for each refinement level.
A computational analysis of the methods proposed in this paper will be subject to future investigation.

\section{Acknowledgments}

This work was supported by the National Science Foundation grant DMS-1412796.

\section*{References}
\bibliography{mgdd}


\end{document}